\documentclass[12pt]{amsart}
\usepackage{cancel}
\usepackage{amssymb}
\usepackage{hyperref}
\usepackage{enumitem}
\usepackage{amsmath}
\usepackage{latexsym}
\usepackage{extarrows}
\usepackage{enumitem}
\usepackage{txfonts}
\usepackage{mathtools}
\usepackage{bm}
\usepackage{tikz-cd}
\usepackage{xcolor}
\usepackage{quiver}
\usepackage{comment}
\usepackage{graphics}
\usepackage{float}
\usepackage{relsize}
\usepackage{bm}
\makeatletter
\@namedef{subjclassname@2020}{%
  \textup{2020} Mathematics Subject Classification}
\makeatother
\usepackage[T1]{fontenc}

\hypersetup{
    hidelinks,
    linktoc=page,
    }

\counterwithin{figure}{section}
\numberwithin{equation}{section}

\theoremstyle{plain}

\newtheorem{theorem}{Theorem}[section]
\newtheorem{proposition}[theorem]{Proposition}

\newtheorem{lemma}[theorem]{Lemma}

\theoremstyle{definition}
\newtheorem{definition}[theorem]{Definition}

\newtheorem{remark}[theorem]{Remark}

\newtheorem{example}[theorem]{Example}

\newtheorem{question}[subsection]{Question}

\newcommand\E{\mathbb{E}}
\newcommand\Z{\mathbb{Z}}

\newcommand\N{\mathbb{N}}

\newcommand\YY{\mathrm{\bf Y}}
\newcommand\XX{\mathrm{\bf X}}

\newcommand\IP{\mathrm{IP}}
\newcommand\EIP{\E_{\gamma\in \IP_{\Phi_N}(\gn)}}

\newcommand{\sEIP}{\E_{\gamma\in \IP_{\Phi_N}}}

\newcommand{\dAPk}{d_{\mathrm{HP}_k}}
\newcommand{\dAP}{d_{\mathrm{HP}_3}}

\newcommand{\gn}{(\gamma_n)_{n\in \N}}
\newcommand{\rec}{\mathrm{rec}}
\newcommand{\norm}[1]{\left\lVert#1\right\rVert}
\newcommand{\ign}{(i\cdot\gamma_n)_{n\in\mathbb{N}}}

\begin{document}
\newcommand{\cdAPk}{\mathrm{conv}\dAPk}


\baselineskip=17pt


\title[Furstenberg-Katznelson constant over finite fields]{On the Furstenberg-Katznelson constant for the IP Szemer\'edi theorem over finite fields}

\author[O. Shalom]{Or Shalom}
\address{Department of Mathematics\\ Bar Ilan University \\ 
Ramat Gan \\
5290002, Israel}
\email{Or.Shalom@math.biu.ac.il}

\date{\today}

\begin{abstract} 
In \cite{Bergelson2000aspects}, Bergelson et al. observed that Furstenberg's proof of Szemer\'edi's theorem provides a positive lower bound on the density of arithmetic progressions in sets of positive density in the integers. Namely, for every $\delta\in(0,1]$ and every $k\in \mathbb{N}$, there exists a positive constant $c=c(k,\delta)>0$ such that
$$\{n\in \mathbb{N} : d(E\cap (E-n)\cap\dots\cap (E-(k-1)n))>c(k,\delta)\} \neq \emptyset$$ whenever $d(E)\ge \delta$. Similarly, in  \cite{FK85}, Furstenberg and Katznelson proved the IP Szemer\'edi theorem, establishing in particular the existence of a constant $c_{\IP}=c_{\IP}(k,\delta)>0$ such that 
$$\{n\in \mathbb{N} :  d(E\cap (E-n)\cap\dots\cap (E-(k-1)n))>c_{\IP}(k,\delta)\}$$ is $\IP^*$ whenever $d(E)\ge \delta$. In this paper, we study analogues of $c$ and $c_{\IP}$ and their ergodic-theoretic counterparts, $c^{\mathrm{rec}}$ and $c_{\IP}^{\rec}$, for vector spaces over finite fields. We provide a qualitative result (Theorem \ref{qualitativemain}) and in special cases such as Roth's theorem and the IP-Roth theorem, we provide strong quantitative bounds (Theorem \ref{quantitative}) for these constants. Our tools are primarily ergodic theoretic; we study the characteristic factors and limit of multiple ergodic averages along $\IP$s in vector spaces over finite fields.
\end{abstract}



\maketitle
\section{Introduction}\label{introduction}  
\subsection{Overview}
The Furstenberg and Katznelson \cite{FK85} $\IP$ Szemer\'edi theorem asserts that every set of positive upper Banach density in the integers contains arbitrarily long arithmetic progressions whose common difference lies in any pre-determined IP-set (see Section \ref{FKconstant:sec}). This paper studies the optimal lower bound for the density of such arithmetic progressions within sets of positive density in a vector space over a finite field (see Theorem \ref{qualitativemain} for the main qualitative result and Theorem \ref{quantitative} for our main quantitative results). Even though we obtain new quantitative results, our primary tools are \emph{qualitative} (i.e., ergodic-theoretic). Roughly speaking, using ergodic-theoretic tools, we will show that the Furstenberg-Katznelson constant for the IP Szemer\'edi is equal to a (convex) density function which counts \emph{Hall-Petresco} progressions (see Definition \ref{dAPk:def}). In the case where $k=3$ (i.e., Roth's theorem and IP Roth theorem), this connection allows us to derive our bounds from existing results, namely the work of Fox and Lov\'asz \cite{Foxlovasz} (see Theorem \ref{removal-lemma}), and the lower bound of  Elsholtz et al. \cite{behrendplus}. Unfortunately, the existing bounds for longer progressions ($k>3$) were not suitable for our methods. In Section \ref{questions:sec}, we pose an open question in quantitative additive combinatorics that would be necessary to extend this approach, as well as other questions which arise naturally from our work.\\

The paper is structured as follows:
\begin{itemize}
    \item Section \ref{introduction} surveys the literature on quantitative bounds for Szemer\'edi's theorem, both for the integers and their finite field analogues.
    \item Section \ref{constants} introduces several constants studied in this paper. 
    \item Section \ref{results} formulates our main results. These include a qualitative result in the general case and a quantitative result for the $k=3$ case (i.e., Roth's theorem and the IP Roth theorem). In this section, we also prove our quantitative result, assuming the qualitative one.
    \item Section \ref{questions:sec} poses several open questions arising from this research. We particularly emphasize a question that, if resolved, could lead to quantitative bounds for progressions of length $k>3$. Such bounds appear to be beyond the reach of purely ergodic-theoretic methods.
\end{itemize}
The subsequent sections are devoted to the proof of our main qualitative result. Section \ref{IPergodicavg:sec} introduces ergodic averages along IPs (previously defined in \cite{SK}, for the integers), and Sections \ref{characteristicfactors:sec} and \ref{formula:sec} study their limit behavior. Finally, Section \ref{proof} contains the proof of our main qualitative result. The Appendix details combinatorial properties of these constants, which are derived from arguments independent of our main ergodic-theoretic strategy.
\subsection{On Szemer\'edi's theorem}
At the heart of additive combinatorics lies the study of arithmetic progressions. In the integers, a \emph{$k$-term arithmetic progression} (or $k$-AP for short) is a set of the form $\{a,a+d,a+2d,\dots,a+(k-1)d\}$ for some $a\in\mathbb{Z}$ and $d\in \mathbb{Z}$. A $k$-AP is called \emph{non-trivial} if $d\neq 0$. A central question, posed by Erd\H{o}s and Tur\'an in 1936 \cite{Erdosturan}, asks whether a \emph{large} set of integers must necessarily contain such arithmetic progressions. To quantify what \emph{large} means, we define a key function $r_k:\mathbb{N}\rightarrow[0,1]$ to be the maximum possible density (i.e., $\frac{|A|}{N}$) of a subset $A \subseteq \{1, 2, \dots, N\}$ that contains no non-trivial $k$-APs\footnote{In the literature, $r_k(N)$ is sometimes defined as the maximal cardinality, rather than density.}. Erd\H{o}s and Tur\'an conjectured that $r_k(N) = o_{k,N\rightarrow\infty}(1)$. This conjecture was resolved affirmatively, first by Roth \cite{Roth1952} for $k=3$, and in its full generality by Szemer\'edi \cite{szemeredi1975sets}. 
\begin{theorem}[Szemer\'edi's Theorem, 1975]
For any integer $k \ge 3$,
$$\lim_{N \to \infty} r_k(N) = 0.$$
\end{theorem}

The study of the exact quantitative behavior of $r_k(N)$ is a central open problem in additive combinatorics.\footnote{This problem, posed by Erd\H{o}s, currently has a \$10,000 prize.} When $k=3$, Behrend \cite{Behrend} established a lower bound for $r_3(N)$. He showed that for every $N$, there exists a set $A \subseteq \{1, 2, \dots, N\}$ of cardinality $|A| \ge N \cdot \exp(-c\sqrt{\log N})$, containing no three-term arithmetic progressions, where $c>0$ is an absolute constant. Hence, $r_3(N) \ge \exp(-c\sqrt{\log N})$. Upper bounds for $r_3(N)$ were initially obtained by Roth \cite{Roth1952} in 1953, who showed that $r_3(N)=O\left(\frac{1}{\log \log N}\right)$. Several improvements were obtained afterwards (see, e.g.,~Szemer\'edi \cite{Szemeredi2}, Bourgain \cite{Bourgaintriplet,Bourgainrevisit}, and Sanders \cite{Sanders}), until eventually Bloom and Sisask \cite{BS1} were able to achieve the upper bound $r_3(N) =O\left(\frac{1}{\log N^{1+c}}\right)$, breaking the so-called \emph{logarithmic barrier}. Finally, in a major breakthrough, Kelley and Meka \cite{KM} provided a new bound, much closer to Behrend's lower bound, showing that $r_3(N)\le \exp\left(-c(\log N)^{1/12}\right)$ for some $c>0$.\footnote{Recently, Bloom and Sisask \cite{BS2} slightly improved the exponent further to $1/9$, and Raghavan \cite{Raghavan} subsequently improved it to almost $1/6$, pushing it closer to the $1/2$ exponent of Behrend's lower bound.} For arithmetic progressions of length $k\ge 4$, quantitative bounds were first established by Gowers \cite{gowers1,gowers2} and were recently improved in special cases (e.g.,~for $k=4$ see \cite{green-tao-r4} and for $k=5$ see \cite{leng2024improved}).
\subsection{The finite field analogue}
The finite field analogue of the questions introduced above is particularly relevant to this paper. Throughout, we fix a prime $p$ and let $V = \mathbb{F}_p^n$ for some integer $n \ge 1$. We are primarily interested in the asymptotic regime where $p$ is fixed and $n \to \infty$. As in the integer setting, a $k$-term arithmetic progression (or $k$-AP) in $V$ is a set of the form $\{a, a+d, \dots, a+(k-1)d\}$ for $a, d \in V$. A $k$-AP is non-trivial if $d \neq 0$. When $3 \le k \le p$,\footnote{All arithmetic progressions in $V$ are of length at most $p$, because $a+pd = a$ for all $a, d \in V$. We thus assume $k\le p$.} we define $r_k(\mathbb{F}_p^n)$ as the maximal density $\frac{|A|}{|V|}$ of a subset $A \subseteq \mathbb{F}_p^n$ containing no non-trivial $k$-APs. An analogue of Behrend's construction in finite fields was recently improved in \cite{behrendplus} and is particularly important in this paper.
\begin{theorem}[Improved Behrend construction over finite fields \cite{behrendplus}]\label{behrendplus:theorem}
   There exists a constant $c > 1/2$ such that for every prime $p$ and every sufficiently large positive integer $n$ (sufficiently large in terms of $p$), there exists a subset $A\subseteq \mathbb{F}_p^n$ of size $\ge (c\cdot p)^n$, containing no non-trivial $3$-APs. In other words, $r_3(\mathbb{F}_p^n)\ge  c^n$.
\end{theorem}
This result will be used as a black box in order to derive strong upper bounds for our main quantitative result (Theorem \ref{quantitative}). In the other direction, we have the upper bound $r_3(\mathbb{F}_3^n) \le \left(\frac{2.756\dots}{3}\right)^n$, established by Ellenberg and Gijswijt \cite{EG} when $p=3$, using the polynomial method. For general $p>3$, the best known upper bound follows from the lower bound of Fox and Lov\'asz for Green's removal lemma \cite{Foxlovasz}. More specifically, in our setting, Fox and Lov\'asz established the following result (see \cite[Theorem 1.4]{KM} for this precise formulation).
\begin{theorem}\label{removal-lemma}
    Let $p\ge 3$. There exists a constant $C_p>0$, such that for all $n\in\mathbb{N}$, if $A\subseteq \mathbb{F}_p^n$ is a set of density $\delta$, then $$\E_{x,a\in \mathbb{F}_p^n}1_A(x)1_A(x+a)1_A(x+2a)\ge \left(\frac{\delta}{3}\right)^{C_p}.$$ Furthermore, $C_p=1+\frac{1}{c_p}$, where $c_p$ satisfies
    $$p^{1-c_p} = \inf_{0<x<1} x^{-(p-1)/3}(x^0+x^1+\dots+x^{p-1}).$$
\end{theorem}
This result will also be used as a black box in the proof of our main quantitative result (Theorem \ref{quantitative}).
A key consequence of this theorem is the upper bound on $r_3(\mathbb{F}_p^n)$. By applying this theorem to a $3$-AP-free set (whose only $3$-APs are trivial), Fox and Lov\'asz derived the bound $r_3(\mathbb{F}_p^n) = O_p(p^{-c_pn})$. Bounds for arithmetic progressions of longer lengths are also known, but are far from optimal and are not directly used in this paper. Instead, in Section \ref{questions:sec}, we propose a new question about the quantitative behavior of some higher-order arithmetic progressions known as \emph{Hall-Petresco progressions}, that arise naturally from our work (see Question \ref{HPquantitative:question}).
\section{Introducing the constants}\label{constants}
In this section, we introduce several constants that we study in this paper. The Furstenberg constant is introduced in Definition \ref{F:constant}, the Furstenberg-Katznelson constant in Definition \ref{FKconstant:def}, and their dynamical counterparts in Definition \ref{dynamical:constants}.
\subsection{Furstenberg constant}
Let $\Gamma = (\Gamma,+)$ denote a countable abelian group. A \emph{F\o lner sequence} for $\Gamma$ is a sequence of finite subsets $\Phi=(\Phi_N)_{N\in\mathbb{N}}$ satisfying the property that
\begin{equation}\label{folner}
    \lim_{N\rightarrow\infty} \frac{|(\Phi_N+a) \triangle \Phi_N|}{|\Phi_N|}=0,
\end{equation}
for all $a\in \Gamma$, where $\Phi_N+a = \{\gamma+a : \gamma\in \Phi_N\}$ and $\triangle$ denotes the symmetric difference. Given a F\o lner sequence $\Phi=(\Phi_N)_{N\in\mathbb{N}}$, we define the \emph{upper density} of a set $E\subseteq \Gamma$ along $\Phi$ by
\begin{equation}\label{density}
    d_{\Phi}(E) = \limsup_{N\rightarrow\infty} \frac{|E\cap \Phi_N|}{|\Phi_N|}.
\end{equation}
We say that a set $E\subseteq \Gamma$ has \emph{positive upper Banach density} if $d^*(E)=\sup_{\Phi} d_{\Phi}(E)>0$, where the supremum ranges over all F\o lner sequences of $\Gamma$. In 1977, Furstenberg \cite{furstenberg1977ergodic} gave an ergodic-theoretic proof of Szemer\'edi's theorem. In fact, Furstenberg proved the following stronger result.
\begin{theorem}[Furstenberg-Szemer\'edi's theorem]\label{FS}
    Let $k\ge 1$ be arbitrary. For every $\delta\in(0,1]$, there exists a positive constant $c=c(k,\delta,\mathbb{Z})>0$ with the property that \begin{equation}\label{lowerbound}\left\{n\in \mathbb{Z}\setminus\{0\} : d_{\Phi}\left(\bigcap_{i=0}^{k-1} (E-i\cdot n)\right)>c\right\}
    \end{equation} is syndetic\footnote{A subset $A\subseteq \Gamma$ of an abelian group $\Gamma$ is called \emph{syndetic} if there exists a finite set $F\subseteq \Gamma$ such that $A+F=\Gamma$.} whenever $\Phi=(\Phi_N)_{N\in\mathbb{N}}$ is a F\o lner sequence and $E\subseteq \mathbb{Z}$ is a set of density $d_{\Phi}(E)>\delta$.
\end{theorem}
To be precise, Furstenberg's original proof only establishes that \eqref{lowerbound} is non-empty (in fact syndetic) when $c=0$, while the existence of a positive constant $c>0$ depending only on $k$ and $\delta$ was later observed by Bergelson et al.\ in \cite{Bergelson2000aspects}. Furstenberg's result is a significant strengthening of Szemer\'edi's theorem. The original theorem guarantees a non-empty intersection, whereas Theorem \ref{FS} guarantees an intersection with positive density. 

The \emph{Furstenberg constant} is the optimal (largest) constant satisfying \eqref{lowerbound}. Formally, and more generally, we define the \emph{Furstenberg constant} as follows.
\begin{definition}[The Furstenberg constants]\label{F:constant}
    Let $k\ge 1$, let $\delta\in[0,1]$ and let $\Gamma$ be a countable abelian group. The \emph{Furstenberg} constant $c=c(k,\delta,\Gamma)$ is defined by the formula
    $$c(k,\delta,\Gamma) = \inf_{\Phi} \inf_{d_{\Phi}(E)\ge \delta} \sup_{\gamma\in \Gamma\setminus\{0\}} d_{\Phi}\left(\bigcap_{i=0}^{k-1}(E-i\gamma)\right),$$ where the infimum is over all F\o lner sequences $\Phi=(\Phi_N)_{N\in\mathbb{N}}$ for $\Gamma$ and all $E\subseteq \Gamma$. Equivalently, $c$ is the largest constant such that
    $$\left\{\gamma\in \Gamma\setminus\{0\} : d_{\Phi}\left(E\cap (E-\gamma)\cap\dots\cap (E-(k-1)\gamma)\right) \ge c(k,\delta,\Gamma)-\varepsilon\right\}\neq  \emptyset$$ for every $\varepsilon>0$,  whenever $d_{\Phi}(E)\ge \delta$.
\end{definition}
The exact dependence of $c(k,\delta,\mathbb{Z})$ on $\delta$ is not currently known in general. An interesting upper bound is $c(k,\delta,\mathbb{Z})\le \delta^k$. This bound can be obtained by considering a \emph{random} subset $E$ (see Lemma \ref{randombound}, and Appendix \ref{randomsec} for the exact formulation). In this case, we almost surely have $d_\Phi\left(\bigcap_{i=0}^{k-1}(E-i\cdot n)\right)=\delta^k$, for all $k$ and all $n\in\mathbb{Z}\setminus\{0\}$. The upper bound $\delta^k$ is optimal for $k=1,2$, with $c(1,\delta,\mathbb{Z})=\delta$ being obvious, and $c(2,\delta,\mathbb{Z}) = \delta^2$ essentially due to Khintchine \cite{khintchine}. However, adapting Behrend's construction, it can be shown that $c(3,\delta,\mathbb{Z})<\delta^{-t\cdot \log \delta}$ for arbitrarily small values of $\delta$, for some absolute constant $t>0$. Unfortunately, for small values of $\delta$, this upper bound is strictly smaller than any polynomial power of $\delta$, let alone $\delta^3$. In \cite{BHK} Bergelson, Host, and Kra considered a different constant. Let $c^*(k,\delta,\mathbb{Z})$ denote the optimal lower bound for \eqref{lowerbound}, defined precisely as $c(k,\delta,\mathbb{Z})$ is defined in Definition \ref{F:constant}, but where every instance of $d_{\Phi}$ is replaced with $d^*$. Bergelson, Host, and Kra showed that $c^*(3,\delta,\mathbb{Z})=\delta^3$ and $c^*(4,\delta,\mathbb{Z})=\delta^4$, but this pattern fails to generalize and at $k \ge 5$ we again get the upper bound $c^*(5,\delta,\mathbb{Z})<\delta^{-t\log \delta}$, for some absolute constant $t>0$. The cases $k=3,4$ of this result were later generalized to the setting of vector spaces over finite fields by Bergelson, Tao and Ziegler \cite{btz2}, and to general abelian groups under certain assumptions \cite{shalom2,ABB,ABS}.
\subsection{Furstenberg and Katznelson constant}\label{FKconstant:sec}
Let $\Gamma=(\Gamma,+)$ be an abelian group and $\gn$ be a sequence in $\Gamma$. The IP set generated by $\gn$ is defined to be the multiset of all finite sums of distinct elements from $\gn$: $$\IP\big(\gn\big) = \left\{\sum_{n \in F} \gamma_n : F \subseteq \mathbb{N}, |F| < \infty \right\},$$ with the convention that the sum over the empty set is $0$.
\begin{example}[Example in the integers]
Let $\Gamma = \mathbb{Z}$ and $\gamma_n = 10^{n-1}$. Then $$\IP\big(\gn\big) = \{0,1,10,11,100,101,110,111,1000,\dots\}$$ is exactly the set of all natural numbers whose decimal representation contains only the digits $0$ and $1$.
\end{example}

\begin{example}[Example over finite fields] Let $\Gamma = \mathbb{F}_p^\omega$ denote the countable-dimensional vector space over $\mathbb{F}_p$ (i.e., the direct sum), and let $\gamma_n = e_n$ denote the $n^{\mathrm{th}}$ element in the standard basis. Then,
$$\IP\big(\gn\big) = \left\{ \sum_{n \in F} e_n : F \subseteq \mathbb{N}, |F| < \infty \right\}.$$ This is the set of all vectors in $\mathbb{F}_p^\omega$ (with finite support) whose coordinates are all either $0$ or $1$.
\end{example}
A subset $A\subseteq\Gamma$ is called $\IP^*$ if for every $\IP$ we have a non-trivial intersection $A\cap (\IP\setminus \{0\}) \neq \emptyset$.\footnote{This technical definition is necessary, because we will adopt the convention that $\IP\setminus \{0\}$ may still contain $0$ if the $\IP$, which is a multiset, contains $0$ more than once.} In \cite{FK85}, Furstenberg and Katznelson proved the IP Szemer\'edi theorem in the integers, which asserts that every subset of the integers with positive upper Banach density contains arbitrarily long arithmetic progressions whose common difference lies in any pre-determined $\IP$. Later, Furstenberg and Katznelson \cite{DHJ3,DHJ} proved the density Hales-Jewett theorem, which can be used to derive an analogue of the IP Szemer\'edi theorem over finite fields.
\begin{proposition}[IP Szemer\'edi theorem over finite fields]
    Let $p$ be a fixed prime and let $V=\mathbb{F}_p^\omega$ denote the countable-dimensional vector space over the finite field $\mathbb{F}_p$ (i.e., the direct sum $\bigoplus_{i\in\omega} \mathbb{F}_p$). Let $\gn\subseteq V$ be an infinite sequence, and let $k\ge 1$. Then for every subset $E\subseteq V$ of positive upper Banach density, there exists a $k$-term arithmetic progression in $E$ whose common difference lies in $\IP\left(\gn\right)\setminus \{0\}$.
\end{proposition}
See also a proof by Zorin \cite{nilpotentIPszemeredi}, which can be applied to more general nilpotent groups.
Unfortunately, neither the Furstenberg and Katznelson argument nor the other proofs of the proposition above tell us much about the quantitative nature of the maximal density of such arithmetic progressions (namely, the Furstenberg-Katznelson constant).
Below, we define this constant in the generality of arbitrary countable abelian groups.
\begin{definition}[Furstenberg-Katznelson constant]\label{FKconstant:def}
    Let $k\ge 1$, let $\delta\in[0,1]$, and let $\Gamma$ be a countable abelian group. The \emph{Furstenberg-Katznelson constant} $c_{\IP}=c_{\IP}(k,\delta,\Gamma)$ is defined by
    $$c_{\IP}(k,\delta,\Gamma) = \inf_{\Phi} \inf_{d_{\Phi}(E)\ge \delta} \inf_{\gn} \sup_{\gamma\in \IP(\gn)\setminus\{0\}} d_{\Phi}\left(\bigcap_{i=0}^{k-1} (E-i\gamma)\right),$$ where the infimum is over all F\o lner sequences, all subsets $E\subseteq \Gamma$, and all sequences $\gn\subseteq \Gamma$. Equivalently, $c_{\IP}$ is the optimal constant such that the set
    $$\left\{\gamma\in \Gamma : d_{\Phi}(E\cap (E-\gamma)\cap\dots\cap (E-(k-1)\gamma)) \ge c_{\IP}(k,\delta,\Gamma)-\varepsilon\right\}$$
    is $\IP^*$, for every $\varepsilon>0$, whenever $d_{\Phi}(E)\ge \delta$. 
\end{definition}
\subsection{Dynamical constants}
Throughout, $\Gamma$ is a countable abelian group. In this section we introduce dynamical constants $c^{\rec}$ and $c_{\IP}^{\rec}$, which correspond to the combinatorial constants $c$ and $c_{\IP}$ via the Furstenberg correspondence principle and its inverse.
\begin{definition}[$\Gamma$-system]
     A \emph{$\Gamma$-system} is a quadruple $\XX=(X,\mathcal{B},\mu,T)$, where $(X,\mathcal{B},\mu)$ is a regular probability space\footnote{Meaning that $X$ is a compact metric space, $\mathcal{B}$ is the Borel $\sigma$-algebra, and $\mu$ is a regular Borel measure.}, and $T:\Gamma\rightarrow \mathrm{Aut}(X,\mathcal{B},\mu)$ is an action of $\Gamma$ on $X$ by measure-preserving transformations $T^\gamma:X\rightarrow X$.  A factor of a $\Gamma$-system $\XX=(X,\mathcal{B}_X,\mu_X,T_X)$ is a $\Gamma$-system $\YY=(Y,\mathcal{B}_Y,\mu_Y,T_Y)$ and a factor map $\pi:X\rightarrow Y$ such that $\pi\circ T_X = T_Y\circ \pi$ a.e., and $\mu_X(\pi^{-1}(A)) = \mu_Y(A)$ for every measurable set $A\subseteq Y$. Given a function $f\in L^2(\mu_X)$, we let $\E(f\mid Y)\in L^2(\mu_Y)$ denote the \emph{conditional expectation} with respect to the factor $Y$; similarly, we let $\E(f\mid \mathcal{B}_Y)\in L^2(\mu_X)$ denote the lift of $\E(f\mid Y)$ to $\XX$ via $\pi$. Every $\Gamma$-system $\XX$ gives rise to a unitary action of $\Gamma$ on $L^2(\mu)$ defined by $\gamma\mapsto T^\gamma$, where here we abuse notation and set $T^\gamma (f) = f\circ T^{\gamma}$. A $\Gamma$-system $\XX$ is called \emph{ergodic} if every invariant function $f\in L^2(\mu)$ is a constant.
\end{definition}
 The Furstenberg correspondence principle was introduced by Furstenberg (e.g.,~in \cite{furstenberg1977ergodic}) and is often used to solve combinatorial questions with ergodic-theoretic tools. Below, we formulate a general version of the correspondence principle for actions of arbitrary countable abelian groups.
\begin{theorem}[The Furstenberg correspondence principle]\label{FCP}
    Let $\Gamma$ be a countable abelian group, let $\Phi=(\Phi_N)_{N\in\mathbb{N}}$ be a F\o lner sequence, and let $E\subseteq \Gamma$. Then there exists a $\Gamma$-system $\XX = (X,\mathcal{B},\mu,T)$ and a measurable set $A\subseteq X$ such that the following holds:
    \begin{itemize}
        \item[(i)] $\mu(A) = d_{\Phi}(E)$.
        \item[(ii)] For every $k\ge 1$ and every $\gamma_0,\gamma_1,\dots,\gamma_k\in \Gamma$, we have\begin{equation}\label{corrinequality}d_{\Phi}\left(\bigcap_{i=0}^k(E-\gamma_i)\right)\ge \mu\left(\bigcap_{i=0}^{k} T^{-\gamma_i}A\right).
        \end{equation}
    \end{itemize}
\end{theorem}
\begin{proof}
    The proof in this generality is given in Joel Moreira's math blog \cite{Moreirablog}. For the sake of completeness, we include it here. Consider the compact space $X_0 = \{0,1\}^\Gamma$ (with respect to the product topology), and let $\zeta\in X_0$ denote the characteristic function of $E$. The group $\Gamma$ acts on $X_0$ by shifts (i.e., $T^\gamma x(\gamma') = x(\gamma+\gamma')$). We let $X$ denote the orbit closure of $\zeta$ under the action of $\Gamma$. Let $N_j$ be a sequence such that 
    $$d_{\Phi}(E) = \lim_{j\rightarrow\infty} \frac{|E\cap \Phi_{N_j}|}{|\Phi_{N_j}|}.$$
We equip $X$ with the induced topology and the Borel $\sigma$-algebra. For each $j\ge 1$, we define a probability measure $\mu_j = \E_{\gamma\in \Phi_{N_j}} \delta_{T^\gamma \zeta}$ on $\XX$, where $\delta_{T^\gamma \zeta}$ is the Dirac measure supported at $T^\gamma \zeta$. By the Banach-Alaoglu theorem, the sequence $\mu_j$ admits a convergent subsequence $\mu_{j_{\ell}}$ in the weak$^*$ topology, whose limit we denote by $\mu$. Let $A\subseteq X$ denote the cylinder set containing all sequences $x\in X$ such that $x(0)=1$. Hence,
\begin{equation}\label{characteristic}
    T^\gamma \zeta\in A \iff \zeta(\gamma)=1\iff \gamma\in E.
\end{equation}
We conclude that $$\mu(A) = \lim_{\ell\rightarrow\infty} \mu_{j_{\ell}}(A) = \lim_{\ell\rightarrow\infty} \E_{\gamma\in \Phi_{N_{j_{\ell}}}} 1_E(\gamma)=d_{\Phi}(E).$$ This proves $(i)$. For $(ii)$, observe that from \eqref{characteristic} we have
\begin{align*}
  d_{\Phi}\left(\bigcap_{i=0}^k (E-\gamma_i)\right) &\ge \lim_{\ell\rightarrow\infty}   \frac{\left|\Phi_{N_{j_{\ell}}}\cap \bigcap_{i=0}^k (E-\gamma_i)\right|}{|\Phi_{N_{j_{\ell}}}|} \\&= \lim_{\ell\rightarrow\infty}\mu_{j_{\ell}}\left(\bigcap_{i=0}^{k} T^{-\gamma_i}A\right) = \mu\left(\bigcap_{i=0}^{k} T^{-\gamma_i}A\right).
\end{align*}

\end{proof}
This theorem tells us that both $c$ and $c_{\IP}$ are closely related to the following dynamical constants.
\begin{definition}[The recurrence constants]\label{dynamical:constants}
    Let $k\ge 1$, let $\delta\in[0,1]$, and let $\Gamma$ be a countable abelian group. 
    \begin{itemize}
        \item[(i)] Let $c^\rec=c^\rec(k,\delta,\Gamma)$ denote the constant
        $$c^{\rec}(k,\delta,\Gamma) = \inf_\XX \inf_{\mu(A)\ge \delta} \sup_{\gamma\in \Gamma\setminus \{0\}} \mu\left(\bigcap_{i=0}^{k-1} T^{-i\gamma} A\right),$$ where the infimum is over all $\Gamma$-systems $\XX=(X,\mathcal{B},\mu,T)$ and all measurable sets $A\subseteq X$. Equivalently, $c^{\rec}$ is the optimal constant such that $$\left\{\gamma\in \Gamma\setminus \{0\} : \mu(A\cap T^{-\gamma}A\cap\dots\cap T^{-(k-1)\gamma} A) \ge c^{\rec}(k,\delta,\Gamma)-\varepsilon\right\} \neq \emptyset$$ for every $\varepsilon>0$, whenever $A$ is a measurable subset of a $\Gamma$-system $\XX=(X,\mathcal{B},\mu,T)$ with $\mu(A)\ge \delta$.
        \item[(ii)] Let $c^\rec_{\IP}=c_{\IP}^{\rec}(k,\delta,\Gamma)$ denote the constant
        $$c_{\IP}^{\rec}(k,\delta,\Gamma) = \inf_{\XX} \inf_{\mu(A)\ge \delta} \inf_{\gn} \sup_{\gamma\in\IP(\gn)\setminus\{0\}} \mu\left(\bigcap_{i=0}^{k-1} T^{-i\gamma} A\right),$$ where the infimum is over all $\Gamma$-systems $\XX=(X,\mathcal{B},\mu,T)$, all measurable sets $A\subseteq X$, and all sequences $\gn\subseteq \Gamma$. Equivalently, $c_{\IP}^{\rec}$ is an optimal constant such that the set
        $$\left\{\gamma\in \Gamma : \mu(A\cap T^{-\gamma} A\cap\dots\cap T^{-(k-1)\gamma} A)\ge c_{\IP}^{\rec}(k,\delta,\Gamma)-\varepsilon\right\}$$ is $\IP^*$ for all $\varepsilon>0$, whenever $A$ is a measurable subset of a $\Gamma$-system $\XX=(X,\mathcal{B},\mu,T)$ with $\mu(A)\ge \delta$. 
    \end{itemize}
\end{definition}
From Theorem \ref{FCP} and the definition, the following inequalities are immediate: $c_{\IP}^{\rec} \le c_{\IP}\le c$ and $c_{\IP}^{\rec}\le c^{\rec}\le c$.  However, we also have the following inverse of the Furstenberg correspondence principle (see \cite{Martin} for this formulation, and \cite{Sohail,Fish} for alternative formulations).
\begin{theorem}\label{inversecorrespondence}
    Let $\Gamma=(\Gamma,+)$ be a countable abelian group, and let $\Phi=(\Phi_N)_{N\in\mathbb{N}}$ be a F\o lner sequence. For every $\Gamma$-system $\XX=(X,\mathcal{B},\mu,T)$ and every $A\in\mathcal{B}$, there exists a subset $E\subseteq \Gamma$ such that for all $k\in\mathbb{N}$ and $\gamma\in \Gamma$ we have 

    $$\mu\left(\bigcap_{i=0}^{k-1}T^{-i\gamma}A\right) = d_{\Phi}\left(\bigcap_{i=0}^{k-1} (E-i\gamma)\right).$$
\end{theorem}
From Theorem \ref{inversecorrespondence}, we deduce that $c_{\IP}^{\mathrm{rec}}=c_{\IP}\le c=c^{\rec}$. Our primary goal is to study these constants in the case where $\Gamma$ is the infinite-dimensional vector space over some finite field $\mathbb{F}_p$. 
\section{Main results}\label{results}
We provide both qualitative and quantitative results.
\subsection{Qualitative results}
Suppose that $\Gamma=V$ is a vector space over some finite field $\mathbb{F}_p$, for some prime $p>2$ that is fixed throughout this section (the case $p=2$ is trivial). For $3\le k \le p$, our qualitative result shows that all of the constants  $c,c_{\IP},c^{\rec}$ and $c_{\IP}^{\rec}$ are equal to some convex density function (in particular, the Furstenberg constant is equal to the Furstenberg-Katznelson constant). Formally, we introduce a density function $\dAPk:[0,1]\rightarrow[0,1]$, which counts \emph{Hall-Petresco progressions}. These progressions appear naturally in the work of Bergelson, Tao, and Ziegler \cite{btz2}, where they are used to derive a limit formula for multiple ergodic averages in the context of vector spaces over finite fields. The Hall-Petresco groups were defined in \cite[Definition 1.7]{btz2}. A function $P:\mathbb{Z}\rightarrow U$ into an abelian group $U$ is called a polynomial of degree $<k$ if it takes the form $P(n) = \sum_{i=0}^{k-1} \binom{n}{i}\cdot u_i$, where $u_i\in U$ for all $i=0,\dots,k-1$.
\begin{definition}[Hall-Petresco groups]
    Let $p$ be a prime, and let $U_1,\dots,U_m$ be compact abelian $p$-torsion   groups\footnote{Throughout, all groups are implicitly assumed to be metrizable.} for some $0 \le m < p$. Let $1 \le k < p$, and let $c_0,\dots,c_k\in \mathbb{F}_p$ be distinct. The Hall-Petresco group
    $$\mathrm{HP}_{c_0,\dots,c_k}(U_1,\dots,U_m)$$ 
    is the closed subgroup of $(U_1\times\dots\times U_m)^{k+1}$ consisting of tuples of the form $(P(c_i))_{i=0}^k$, where $P=(P_1,\dots,P_m):\mathbb{Z}\rightarrow U_1\times\dots\times U_m$, and for each $1 \le j \le m$, $P_j : \mathbb{Z} \rightarrow U_j$ is a polynomial of degree $<j+1$.
\end{definition}
Let $\mathcal{U}=U_1\times\dots\times U_m$ and for every $j$ set $\mathcal{U}_j = \prod_{i\ge j} U_i$ and $\mathcal{U}_0=\mathcal{U}$. The most important example is the group $$\mathrm{HP}_{0,1,2,\dots,k}(U_1,\dots,U_m) = \left\{\left(u_0,u_0+u_1,\dots,u_0 +ku_1 + \binom{k}{2}u_2+\dots+\binom{k}{m}u_m\right) : \forall_i \text{ } u_i\in \mathcal{U}_i \right\},$$ with the convention that $\binom{k}{m}=0$ whenever $m>k$. Sometimes we will also be interested in the \emph{trivial} Hall-Petresco subgroup
$$\mathrm{HP}_{0,0,\dots,0}(U_1,\dots,U_m) =\{(u,u,\dots,u) : u\in \mathcal{U}\}$$ which consists of the trivial progressions of length $k+1$.
Since the sum of polynomials is a polynomial, it is easy to verify that the Hall-Petresco group is a compact abelian group. Therefore, it admits a Haar measure $\mu_{\mathrm{HP}_{0,1,2,\dots,k}(U_1,\dots,U_m)}$. We can now define our main density function.
\begin{definition}[Hall-Petresco progressions density function]\label{dAPk:def}
    Let $U_1,\dots,U_m$ be compact abelian $p$-torsion groups. For every measurable $f:U_1\times U_2\times\dots\times U_m\rightarrow [0,1]$, we define
    $$S_{U_1,\dots,U_m}^k(f) := \int_{\mathrm{HP}_{0,1,\dots,k-1}(U_1,\dots,U_m)} \bigotimes_{i=0}^{k-1} f\,d\mu_{\mathrm{HP}_{0,1,\dots,k-1}(U_1,\dots,U_m)}.$$
    For a fixed $\delta\in[0,1]$, we let $\dAPk(\delta)$ denote the infimum of $S_{U_1,\dots,U_m}^k(f),$ ranging over all compact abelian $p$-torsion groups $U_1,\dots,U_m$, and all $f:U_1\times U_2\times\dots\times U_m\rightarrow[0,1]$ with $\int_{U_1\times\dots\times U_m} f\,d\mu_{U_1\times\dots\times U_{m}}\ge \delta$. 
\end{definition}
\begin{remark}
In the definition above we allow $f$ to be an arbitrary measurable function with values in $[0,1]$. This freedom will be useful somewhere later in our proof. However, we note that in Lemma \ref{function-set} below we show that the value of $\dAPk(\delta)$ remains the same even if we assume that $f$ must be a characteristic function of a subset of $U_1\times\dots\times U_m$.
\end{remark}

The exact dependence of $\dAPk$ on $\delta$ is not known (see Question \ref{HPquantitative:question}). However, in the special case where $k=3$, Hall-Petresco progressions are $3$-APs and bounds are known when the groups $U_1,\dots,U_m$ in the definition above are finite. 

In section \ref{proof}, we will show that the constants $c_{\IP}, c^{\rec}, c_{\IP}^{\rec}$ and $c$ are all convex in $\delta$, but we do not know if this holds for $\dAPk$ (see Question \ref{convex:question}). Thus, we will work with the lower convex envelope $\cdAPk$ of $\dAPk$. Namely, \begin{equation}\label{convex}
    \cdAPk(\delta) = \sup\{f(\delta):f \text{ is convex and }f\le \dAPk  \text{ on }[0,1]\}
\end{equation} is the largest convex function that is pointwise less than or equal to $\dAPk$. 

We are ready to formulate our main qualitative result.
\begin{theorem}[Qualitative correspondence for the Furstenberg and Furstenberg-Katznelson constants]\label{qualitativemain}
    For every $k\ge 3$ and every $\delta\in[0,1]$, we have
    \begin{equation}\label{inequality:main}
       c_{\IP}^{\rec}(k,\delta,\mathbb{F}_p^\omega)= c_{\IP}(k,\delta,\mathbb{F}_p^\omega)= c^{\rec}(k,\delta,\mathbb{F}_p^\omega) =c(k,\delta,\mathbb{F}_p^\omega)= \cdAPk(\delta).
    \end{equation}
\end{theorem}
This theorem can be viewed as a correspondence principle in the spirit of Furstenberg. It relates the combinatorial problem of counting arithmetic configurations in subsets of a discrete, countable-dimensional vector space $V$ to the problem of measuring Hall-Petresco progressions in compact abelian groups. 
\subsection{Quantitative results}
In some special cases, we are also able to study the dependence of the constants defined above on $\delta$. Simple instances are the cases where $k=1$ and $k=2$. In these cases, we are able to compute all of the constants explicitly.
\begin{lemma}[The simple cases]\label{simple}
    Let $k=1$ or $k=2$. Then for all $\delta\in[0,1]$ and all primes $p>2$ we have,
     $$c_{\IP}^{\rec}(k,\delta,\mathbb{F}_p^\omega)= c_{\IP}(k,\delta,\mathbb{F}_p^\omega)= c^{\rec}(k,\delta,\mathbb{F}_p^\omega) = c(k,\delta,\mathbb{F}_p^\omega) = \cdAPk(\delta) =\delta^k.$$
\end{lemma}
Some of these results have already been established in the literature (see, e.g.,~\cite{khintchine,bergelson2010ultrafilters}) in the setting of the integers. 

Our main new quantitative result is in the case where $k=3$. 
\begin{theorem}[Quantitative bounds for the Furstenberg-Katznelson IP Roth theorem over finite fields]\label{quantitative}
    Let $p>2$ be a fixed prime number, and let $\delta\in[0,1]$. Then there exists a constant $C_p=\Theta(\log p)$ such that 
    $$\left(\frac{\delta}{3}\right)^{C_p}\le c_{\IP}(3,\delta,\mathbb{F}_p^\omega)\le  \delta^3.$$ Moreover, there are arbitrarily small values of $\delta$ satisfying the following strong upper bound
    $$c_{\IP}(3,\delta,\mathbb{F}_p^\omega)\le \delta^{1+\log p}+\delta^{2\log p},$$ where here $\log=\log_2$ is the logarithm in base $2$.
\end{theorem}
\begin{remark}
    The strong upper bound in Theorem~\ref{quantitative} is not valid for large values of $\delta$. Indeed, suppose that $\delta=1-\varepsilon$ for $\varepsilon>0$ sufficiently small.
    By the union bound, if $\mu(A),\mu(B),\mu(C)\ge 1-\varepsilon$, then $\mu(A\cap B\cap C)\ge 1-3\varepsilon$. Thus, we must have $c_{\IP}(3,\delta,\mathbb{F}_p^\omega) \ge 1-3\varepsilon$. As $\varepsilon\rightarrow 0$, this lower bound becomes arbitrarily close to $(1-\varepsilon)^3 = \delta^3$. 
\end{remark}
Assuming our main qualitative result (i.e., Theorem \ref{qualitativemain}), we will now prove Theorem \ref{quantitative}. The weak upper bound, $c_{\IP}(3,\delta,\mathbb{F}_p^\omega)\le \delta^3$ follows from the obvious fact that $c_{\IP}(3,\delta,\mathbb{F}_p^\omega)\le c(3,\delta,\mathbb{F}_p^\omega)$ and Lemma \ref{randombound}. The strong upper bound for small values of $\delta$ follows from a variant of Behrend's construction (Theorem \ref{behrendplus:theorem}).
\begin{proof}[Proof of the upper bound in Theorem \ref{quantitative}]
    From Theorem \ref{behrendplus:theorem}, we see that for all $n$ sufficiently large, there exist subsets $A_n\subseteq \mathbb{F}_p^n$ of size $|A_n|=\lceil\left(\frac{p}{2}\right)^n\rceil$ without any non-trivial $3$-APs. For every $\delta\in[0,1]$ we set $n=-\lfloor\log \delta\rfloor$ where $\log$ is in base $2$. When $\delta$ is a power of $2$, we have $\delta\le \mu(A_n)\le \delta+\frac{1}{p^n}$. Because $A_n$ is free of non-trivial $3$-APs, the only solutions to $x, x+a, x+2a \in A_n$ occur when $a=0$. Thus, out of the $p^{2n}$ possible pairs of $(x,a)$, exactly $|A_n|$ of them satisfy the condition. Thus, for $\delta$ sufficiently small such that the statement above is satisfied for $n$ we have, \begin{equation}\label{logp}\E_{a\in \mathbb{F}_p^n}\E_{x\in \mathbb{F}_p^n} 1_{A_n}(x)1_{A_n}(x+a)1_{A_n}(x+2a) = \frac{\mu(A_n)}{p^n}\le \delta^{1+\log p}+\delta^{2\log p}.
    \end{equation}

    On the other hand, consider the system $\XX = \mathbb{F}_p^{\mathbb{N}\times \mathbb{N}}\times \mathbb{F}_p^{\mathbb{N}}$ equipped with the $\mathbb{F}_p^\omega$-action $T^\gamma (x,y) = (x,y+\gamma\cdot x)$ where $\gamma\cdot x = \sum_{i=1}^\infty \gamma_i x^{(i)}$, where $x^{(i)}\in \mathbb{F}_p^\mathbb{N}$ and $\gamma_i\in \mathbb{F}_p$ are the coordinates of $x$ and $\gamma$, respectively. Let $A = \mathbb{F}_p^{\mathbb{N}\times \mathbb{N}}\times (A_n\times \mathbb{F}_p^{\mathbb{N}\setminus \{1,\dots,n\}})$. Because $x$ is uniformly distributed in $\mathbb{F}_p^{\mathbb{N}\times \mathbb{N}}$, for any $\gamma \neq 0$, the projection of $\gamma \cdot x$ onto its first $n$ coordinates acts as a uniformly distributed random variable independent of $y$. Therefore, integrating over the system is equivalent to averaging over all pairs $(x, a) \in \mathbb{F}_p^n \times \mathbb{F}_p^n$. Thus, we have
    $$c^{\rec}(3,\delta,\mathbb{F}_p^\omega)\le \mu(A\cap T^{-\gamma} A\cap T^{-2\gamma} A)=\E_{a\in \mathbb{F}_p^n}\E_{x\in \mathbb{F}_p^n} 1_{A_n}(x)1_{A_n}(x+a)1_{A_n}(x+2a)$$ for all $\gamma\neq 0,$
    and the claim follows from Theorem \ref{qualitativemain}, and \eqref{logp}.
\end{proof} 
\begin{remark}
The bound in Equation \eqref{logp} relies on a 3-AP-free subset of size $\lceil(p/2)^n\rceil$. By fully utilizing the recent result of Elsholtz et al., which guarantees a subset of size at least $(cp)^n$ for an absolute constant $c > 1/2$, the base of the logarithm can be optimized. Setting $n = \lfloor\log_c \delta\rfloor$ yields the strictly tighter upper bound $\delta^{1+\log_{1/c} p} + \delta^{2\log_{1/c} p}$. We maintain the base $2$ formulation in the main proof for expository clarity.
\end{remark}
\begin{remark}
A similar argument, using Behrend type constructions for longer progressions, can be utilized to prove an upper bound for $c_{\IP}(k,\delta,\mathbb{F}_p^\omega)$. For instance, see \cite{behrendlonger} for upper bounds on $r_k(\mathbb{F}_p^n)$. Unfortunately, these bounds are not optimal and consequently, the corresponding bounds on $c_{\IP}$ are also not optimal.
\end{remark}
The proof of the lower bound follows from a lower bound for Green's removal lemma obtained by Fox and Lovatz (Theorem \ref{removal-lemma}). 
\begin{proof}[Proof of the lower bound in Theorem \ref{quantitative}]
Fix a constant $C_p$ as in Theorem \ref{removal-lemma}, and note that $C_p = \Theta(\log p)$ (see \cite{Foxlovasz}). By Theorem \ref{qualitativemain}, and since $\delta\rightarrow(\frac{\delta}{3})^{C_p}$ is convex, it suffices to show that $\dAP(\delta) \ge (\frac{\delta}{3})^{C_p}$, for all $\delta\in[0,1]$.  Assume by contradiction that 
    $$
        \dAP(\delta) < (\frac{\delta}{3})^{C_p}-\varepsilon_0
   $$
    for some $\delta\in[0,1]$ and $\varepsilon_0>0$. By definition, we can find a $p$-torsion group $U$ and a function $f:Z\rightarrow[0,1]$ with $\int_{U} f\,d\mu_U \ge \delta$, and 
     $$\int_U \int_U f(x)f(x+a)f(x+2a) \,d\mu_U(x)\,d\mu_U(a) < (\frac{\delta}{3})^{C_p}-\varepsilon_0,$$ where $\mu_U$ is the Haar measure on $U$.
    Since the continuous functions on $U$ are dense in $L^1(\mu_U)$, we may assume without loss of generality that $f$ is continuous. By the structure theorem of $p$-torsion groups, we can write $U = \prod_{i\in I} \mathbb{F}_p$ for some countable set of indices $I$. If $I$ is finite, we can always extend $U$ to some infinite group and lift $f$ to that extension, so we shall assume without loss of generality that $I=\mathbb{N}$. We can then embed $\mathbb{F}_p^\omega$ in $U$ in the obvious manner.

     Set $F(\gamma) := f(\gamma)$ for all $\gamma\in \mathbb{F}_p^\omega$. Choose the F\o lner sequence $\Phi_N = \mathbb{F}_p^N$. Since $f$ is continuous we have $$\lim_{N\rightarrow\infty}\E_{\gamma\in \Phi_N} F(\gamma)\ge \delta$$ and 
     \begin{align*}
\lim_{N\rightarrow\infty}\E_{a\in \Phi_N} &\E_{x\in \Phi_N}F(x)F(x+a)F(x+2a)\\&= \int_U \int_U f(x)f(x+a)f(x+2a) \,d\mu_U(x)\,d\mu_U(a).   
     \end{align*}
     By Proposition \ref{randomset}, we can find a subset $E\subseteq\mathbb{F}_p^\omega$, such that $d_{\Phi}(E)\ge \delta$ and for all $a\neq 0$, $d_{\Phi}(E\cap E-a\cap E-2a) = \lim_{N\rightarrow\infty} \E_{x\in \Phi_N}F(x)F(x+a)F(x+2a). $ Since this holds for all $a\neq 0$, and $|\Phi_N|\rightarrow\infty$ as $N\rightarrow\infty$, we see that
     $$\lim_{N\rightarrow\infty} \E_{a\in \Phi_N} d_{\Phi}(E\cap E-a\cap E-2a) \le  (\frac{\delta}{3})^{C_p}-\frac{\varepsilon_0}{2}.$$
     We conclude, that for all $N$ sufficiently large and all $M$ sufficiently large with respect to $N$ ($M\gg N$), there exists a subset $E' \subseteq \mathbb{F}_p^M$ satisfying that $|E'|\ge \delta p^M$ and 
     $$\E_{a\in \mathbb{F}_p^N} \E_{x\in \mathbb{F}_p^M} 1_{E'}(x) 1_{E'}(x+a) 1_{E'}(x+2a) \le  (\frac{\delta}{3})^{C_p}-\frac{\varepsilon_0}{2}. $$
    This is almost a contradiction to Theorem \ref{removal-lemma}, but unfortunately, here $M$ may not be equal to $N$. To overcome that, we write $\mathbb{F}_p^M  = \mathbb{F}_p^N \times \mathbb{F}_p^{M-N}$ and for $x\in \mathbb{F}_p^M$ write $x=(x_N,x_{M-N})$ for $x_N\in \mathbb{F}_p^N$ and $x_{M-N}\in \mathbb{F}_p^{M-N}$. Let
    $$E'_{x_{M-N}} = \{x\in \mathbb{F}_p^N : (x,x_{M-N}) \in E'\},$$ 
    and let $\delta_{x_{M-N}} := \frac{|E'_{x_{M-N}}|}{p^N}$. Then,
    \begin{equation}\label{avgdelta}
        \E_{x_{M-N}\in \mathbb{F}_p^{M-N}} \delta_{x_{M-N}} \ge \delta
    \end{equation}
and
\begin{equation}\label{contradiction}\E_{x_{M-N}\in \mathbb{F}_p^{M-N}} \E_{a\in \mathbb{F}_p^N} \E_{x\in \mathbb{F}_p^N} 1_{E'_{x_{M-N}}}(x) \cdot1_{E'_{x_{M-N}}}(x+a)\cdot 1_{E'_{x_{M-N}}}(x+2a) < (\frac{\delta}{3})^{C_p}-\varepsilon_0/2.
    \end{equation}
    However, by Fox and Lov\'asz (Theorem \ref{removal-lemma}), we have that $$\E_{a\in \mathbb{F}_p^N} \E_{x\in \mathbb{F}_p^N} 1_{E'_{x_{M-N}}}(x) \cdot1_{E'_{x_{M-N}}}(x+a)\cdot 1_{E'_{x_{M-N}}}(x+2a) \ge \left(\frac{\delta_{x_{M-N}}}{3}\right)^{C_p}$$ and so by \eqref{avgdelta} and the convexity of $\delta\mapsto \left(\frac{\delta}{3}\right)^{C_p}$, we see that the left hand side in \eqref{contradiction} is bounded below by $$\E_{x_{M-N}\in \mathbb{F}_p^{M-N}} (\frac{\delta_{x_{M-N}}}{3})^{C_p} \ge \left(\frac{\delta}{3}\right)^{C_p},$$ which is a contradiction to \eqref{contradiction}, as required.
\end{proof}
\section{Open Questions and conjectures}\label{questions:sec}
We are interested in generalizing the quantitative results obtained in Theorem \ref{quantitative} to longer progressions $(k>3)$. In this case, the density function $\dAPk$ counts Hall-Petresco progressions, which are more general than the $k$-term arithmetic progressions.  Let $p$ be a fixed prime, let $k,m\in \mathbb{N}$ and let $n_1,\dots,n_m\in \mathbb{N}$. We denote by $r_{k}(\mathbb{F}_p^{n_1},\dots,\mathbb{F}_p^{n_m})$ the maximal density of a subset $A\subseteq\mathbb{F}_p^{n_1}\times\dots\times \mathbb{F}_p^{n_m}$ containing no non-trivial Hall-Petresco progression of length $k$. Namely,
    $$A^k\cap \mathrm{HP}_{0,1,\dots,k-1}(\mathbb{F}_p^{n_1},\dots,\mathbb{F}_p^{n_m}) =A^k\cap \mathrm{HP}_{0,0,\dots,0}(\mathbb{F}_p^{n_1},\dots,\mathbb{F}_p^{n_m}).$$ For every $\delta\in[0,1]$, let $$R_{k,m,p}(\delta) = \inf_{n_1,\dots,n_m\in\mathbb{N}} \inf_{|A| \ge \delta p^{\sum_{i=1}^m n_i}} \E_{\bm{x}\in \mathrm{HP}_{0,1,\dots,k-1}(\mathbb{F}_p^{n_1},\dots,\mathbb{F}_p^{n_m})} \bigotimes_{i=0}^{k-1} 1_A(\bm{x}),$$ where the infimum is over all integers $n_1,\dots,n_m$, and all subsets $A\subseteq \mathbb{F}_p^{n_1}\times\dots\times\mathbb{F}_p^{n_m}$.
    We can generalize a question of Erd\H{o}s to Hall-Petresco progressions.
\begin{question}[Erd\H{o}s question for Hall-Petresco progressions]
    Find the asymptotic behavior of $r_k(\mathbb{F}_p^{n_1},\dots,\mathbb{F}_p^{n_m})$ as a function of $n_1,\dots,n_m$, where $m,k$ and $p$ are fixed.
\end{question}
As usual, a reasonable approach is to provide quantitative bounds for the counting function.
\begin{question}[Quantitative bounds for Hall-Petresco counting function]\label{HPquantitative:question}
Let $m$ and $k$ be fixed integers and $p$ a fixed prime number. Find the asymptotic behavior of $R_{k,m,p}(\delta)$ as a function of $\delta$.
\end{question}
The most interesting instance is when $m=k-1$. Indeed, $\dAPk(\delta) = R_{k,k-1,p}(\delta)$, and so Theorem \ref{qualitativemain} can be used to derive quantitative bounds for the constants $c_{\IP}^{\rec}(k,\delta,\mathbb{F}_p^{\omega}),c_{\IP}(k,\delta,\mathbb{F}_p^{\omega}),c^{\rec}(k,\delta,\mathbb{F}_p^{\omega}),c(k,\delta,\mathbb{F}_p^\omega)$ for general values of $k\ge 3$.

While these constants are convex in $\delta$, unfortunately we do not know the same for $\dAPk$. 
\begin{question}\label{convex:question}
    Let $k\ge 3$, and let $p\geq k$ be a fixed prime. Is $\delta\mapsto \dAPk(\delta)$ convex?
\end{question}

Our last question addresses whether the $\sup$ in the definitions of our constants (see Definitions \ref{F:constant}, \ref{FKconstant:def}, and \ref{dynamical:constants}) is actually achieved as a maximum. Before exploring this, we mention a related result. Consider the constant $c^{\mathrm{wm-rec}}$, which is defined analogously to $c^{\rec}$ but with the additional assumption that the underlying system $\XX$ is weakly mixing. In this case, it is a classical fact that $c^{\mathrm{wm-rec}}(\delta,k,\Gamma)=\delta^k$. In \cite{underrecurrent}, Boshernitzan, Frantzikinakis, and Wierdl showed that there exists an under-recurrent weakly mixing system. Specifically, they constructed a weakly mixing system $\XX$ containing a subset $A$ such that $\mu\left(\bigcap_{i=0}^{k-1}T^{-in}A\right)<\mu(A)^k$ for all nonzero $n\in \mathbb{N}$ (thus justifying the role of $\varepsilon>0$ in this special case). Naturally, we ask whether a variant of this phenomenon holds in our setting across all defined constants.
\begin{question}[Under-recurrence]For every fixed prime $p$, integer $k\le p$, and $\delta\in[0,1]$, does there exist an $\mathbb{F}_p^\omega$-system $\XX=(X,\mathcal{B},\mu,T)$ and a measurable subset $A\subseteq X$ with $\mu(A)=\delta$ such that$$\mu\left(\bigcap_{i=0}^{k-1} T^{-i\gamma} A\right)<\cdAPk(\delta)$$for all $\gamma\in \mathbb{F}_p^\omega\setminus \{0\}$?\end{question}
\subsection*{Acknowledgements}
I would like to thank Dor Elboim for useful discussions about probability theory resulting in a simpler proof of the results in Section~\ref{randomsec}; Dan Florentin for useful discussions about the lower convex envelope; Alexander Fish for useful discussions about the inverse of the Furstenberg correspondence principle. I also thank Bryna Kra for her encouragement. Finally, I acknowledge the use of Gemini 3.1 pro for copy editing. 

\section{Ergodic averages along IPs}\label{IPergodicavg:sec}
Our primary tool is a structural result for the characteristic factors of ergodic averages along $\IP$s. These objects were recently introduced in \cite{SK} in the integer setting. Here we generalize the definition from \cite{SK} to all countable abelian groups $\Gamma$. First, let $\Phi = (\Phi_N)_{N\in\mathbb{N}}$ be a F\o lner sequence of $\N$ and let $\gn$ be a sequence of elements in $\Gamma$. We define the IP-F\o lner sequence associated with $\Phi$ and $\gn$ as the sequence of finite multisets
$$\IP_{\Phi_N}\big(\gn\big)=\{\gamma_{n_1}+\dots+\gamma_{n_{\ell}} : \ell\ge 0, n_1,\dots,n_{\ell}\in\Phi_N \text{ are distinct}\}.$$
We cautiously note that an $\IP$-F\o lner sequence is \textbf{not} a F\o lner sequence in general, even if the multisets are sets.
\begin{definition}[Ergodic average along an $\IP$-F\o lner sequence]
    Let $\Gamma$ be a countable abelian group, let $\XX=(X,\mathcal{B},\mu,T)$ be a $\Gamma$-system, let $\gn$ be a sequence of elements in $\Gamma$, and let $\Phi=(\Phi_N)_{N\in\mathbb{N}}$ be a F\o lner sequence for $\mathbb{N}$. The \emph{ergodic average of $f\in L^2(\mu)$ along the $\IP$-F\o lner sequence $\IP_{\Phi_N}\big(\gn\big)$} is defined by
    \begin{equation}\label{IPergodicavg}
    A_N(f) = \E_{\gamma\in\IP_{\Phi_N}(\gn)}T^\gamma f.
\end{equation}
\end{definition}
\begin{example}
    Let $\Gamma = \mathbb{F}_p^\omega$, and let $(e_n)_{n\in\mathbb{N}}$ be the standard basis for $\Gamma$. Let $\Phi_N = [1,N]$, and let $f\in L^2(\mu)$ be any function on a $\Gamma$-system $\XX=(X,\mathcal{B},\mu,T)$. Then, the ergodic average of $f$ over $\IP_{\Phi_N}\left((e_n)_{n\in \mathbb{N}}\right)$ is
    $$A_N(f) = \frac{1}{2^N}\sum_{\omega\in \{0,1\}^N} T^{\sum_{i=1}^N\omega_ie_i} f.$$ 
\end{example}
The convergence (and divergence) of these ergodic averages were studied in \cite{SK} for $\Z$-systems, and in some special cases the limit was also computed. Here, we will show that the situation in vector spaces over finite fields is somewhat simpler. That is, the ergodic average over any $\IP$ always converges along an \emph{increasing} $\IP$-F\o lner sequence.
We say that a F\o lner sequence $\Phi=(\Phi_N)$ is increasing if it takes the form $\Phi_N = [a,b_N]$ for some fixed $a$ and a sequence $b_N\rightarrow \infty$.
\begin{theorem}[Convergence of the mean ergodic average along increasing $\IP$-F\o lner sequences.]\label{IPMET}
     Let $p$ be a prime number, let $V$ be a countable vector space over $\mathbb{F}_p$, and let $\XX=(X,\mathcal{B},\mu,T)$ be a $V$-system. Let $\gn$ be an arbitrary sequence of elements in $V$, and let $\Phi=(\Phi_N)_{N\in\mathbb{N}}$ be an increasing F\o lner sequence. Then, for every $f\in L^2(\mu)$, the limit
    $$\lim_{N\rightarrow\infty} \EIP T^\gamma f$$ exists in $L^2(\mu)$.
\end{theorem}
\begin{proof}
    Let $\XX$ be as in the theorem, and let $\Sigma\cong \mathbb{F}_p^\mathbb{N}$ denote the Pontryagin dual of $V\cong \mathbb{F}_p^\omega$. The spectral theorem for unitary operators (see, e.g.,~\cite{folland-harmonic-analysis}) implies that there exists a projection-valued Borel measure $\rho$ on $\Sigma$ such that
$$T^\gamma = \int_{\Sigma} \xi(\gamma)~d\rho(\xi)$$ for all $\gamma\in \Gamma$. Thus, for all $N\in \mathbb{N}$ we have
$$\EIP T^\gamma = \int_{\Sigma} \EIP\xi(\gamma)~d\rho(\xi).$$
Let $\phi_N:\Sigma\rightarrow \mathbb{C}$ denote the map $\phi_N(\xi) = \EIP \xi(\gamma)$. We claim that $\lim_{N\rightarrow\infty}\phi_N$ exists pointwise.

Let $\xi\in \Sigma$. The following equality is a direct computation (see, e.g.,~\cite[Proposition 3.2]{SK})
\begin{equation}\label{avgtoprod} \phi_N(\xi)=\sEIP \xi(\gamma) = \prod_{n\in \Phi_N} \left(\frac{1+\xi(\gamma_n)}{2}\right).
\end{equation}
We study two cases, if $\xi(\gamma_n)=1$ for all but finitely many $n\in \mathbb{N}$, then all but finitely many of the elements in the product above are trivial. Thus, the sequence $\phi_N(\xi)$ is eventually constant in $N$ and so $\lim_{N\rightarrow\infty}\phi_N$ exists. Otherwise, $\xi(\gamma_n)\neq 1$ for infinitely many $n\in \mathbb{N}$. In that case $\left|\frac{1+\xi(\gamma_n)}{2}\right| < 1-\varepsilon_p$ for some absolute constant $\varepsilon_p>0$ depending only on $p$. In particular, as $N\rightarrow\infty$, $\Phi_N$ contains an unbounded number of such $n$ and so $$\lim_{N\rightarrow\infty}\sEIP \xi(\gamma)=0.$$ Thus, we can write $\phi=\lim_{N\rightarrow\infty}\phi_N$ and since $\phi_N$ are uniformly bounded, the Lebesgue dominated convergence theorem \cite[Proposition 1.48]{folland-harmonic-analysis} implies that
$$\lim_{N\rightarrow\infty}\sEIP T^\gamma = \int_\Sigma \lim_{N\rightarrow\infty}\sEIP\xi(\gamma)~d\rho(\xi) = \int_\Sigma \phi(\xi)~d\rho(\xi) $$ in the strong topology. Equivalently, $\lim_{N\rightarrow\infty}\sEIP T^\gamma f = \int_{\Sigma} \phi(\xi)~d\rho(\xi)f$ for all $f\in L^2(\mu)$, as required.
\end{proof}
Our next goal is to derive an explicit formula for the limit. Following the notation in the proof of the theorem above, we may approach this problem by computing the limit function $\phi(\xi)$. We can simplify our analysis by observing that the sequence $\gn$, and more generally the $\IP$ generated by $\gn$ are both contained in the subspace $W=\mathrm{Span}_{\mathbb{F}_p}(\gn)$. Restricting our attention to the action of $W$ then reduces the problem to the case where $\gn$ generates the acting group. Moreover, since the action of the complementary subspace $W'$ (i.e., $V=W\oplus W'$) is independent of the $\IP$, we wish to \emph{decompose} the action into a trivial and a generating component. We will formalize this decomposition in the next section.
\subsection{Decomposition of the action}
Throughout, we denote by $V_{\gn}$ the subspace of $V$ generated by the vectors $\gn$. We give two examples that illustrate the two extreme behaviors of the ergodic averages along $\IP\big(\gn\big)$, in the cases where the limit is not the projection onto the $T$-invariant functions. The first example corresponds to a trivial action along the $\IP$ (i.e., the trivial component).
\begin{example}[The trivial component]\label{trivial:ex}
    Let $V$ be an infinite vector space over $\mathbb{F}_p$. Fix a basis  $\mathcal{B}=(e_n)_{n\in\mathbb{N}}$ for $V$ and write $V=\mathbb{F}_p^{\omega}$ in this basis. Let $\gamma_n = e_{2n}$ denote the sequence containing every other element from $\mathcal{B}$. Let $\XX = \mathbb{F}_p^{\mathbb{N}}$ equipped with the product $\sigma$-algebra and Haar measure and consider the action $$(T^v x)_i = x_i + v_{2i-1},$$ where $\mathbb{F}_p^\omega$ is embedded in $\mathbb{F}_p^\mathbb{N}$ in the obvious manner. We see that only the odd coordinates of $v$ are acting non-trivially on $\XX$. In that case, by construction, the action of $T^{\gamma_n}$ on $\XX$ is trivial. Hence, the ergodic average is also trivial (i.e., maps every $f\in L^2(\mu)$ to itself).
\end{example}
We see that the factor of $\XX$ generated by the $V_{\gn}$-invariant functions must play a role in the limit formula for the mean ergodic theorem. We will denote this factor by $Z^0_{\gn}(\XX)$. On the other extreme, we have the \emph{generated component}.
\begin{example}[The generated component]\label{junta:ex}
    Let $X=\mathbb{F}_p^{\mathbb{N}}$ be the infinite product of $\mathbb{F}_p$ and let $V=\mathbb{F}_p^\omega$ embedded in $X$ in the obvious manner. We can then define an action of $V$ on $\XX$ by setting $T^\gamma x=x+\gamma$ for all $\gamma\in V$ and $x\in X$. This time, we let $\gn=(e_n)_{n\in \mathbb{N}}$ denote the standard basis for $V$. Observe that in this case the vectors $\gn$ generate the acting group $V$. Unfortunately, this time the limit of the ergodic average is no longer a projection onto any subspace of $L^2(\mu)$. For instance, let $f:X\rightarrow \mathbb{C}$ denote the projection to the first coordinate (here we identify $\mathbb{F}_p$ with the set of $p^{\mathrm{th}}$-roots of unity in $S^1$). Taking $\Phi_N = [1,N]$, we see that
    $$\EIP T^\gamma f = (\frac{1+\zeta_p}{2})f$$ for all $N$, where $\zeta_p = e^{\frac{2\pi i}{p}}$.
\end{example}
To capture these distinct behaviors formally, we introduce the following definitions.
\begin{definition}[Eigenfunctions, Eigenspaces and Juntas.]
Let $V$ be a countable vector space over $\mathbb{F}_p$, let $\gn$ be a sequence of elements in $V$ and let $\XX= (X,\mathcal{B},\mu,T)$ be a $V$-system.
    \begin{itemize}
        \item[(i)] An \emph{eigenfunction} for $T$ is a non-zero element $f\in L^2(\mu)$ satisfying that $T^\gamma f = \lambda(\gamma) f$ for all $\gamma\in V$, for some character $\lambda:V\rightarrow S^1$. The character $\lambda$ is called the \emph{eigenvalue} of $f$.
        \item[(ii)] For every character $\lambda:V\rightarrow S^1$, we let $L^2_\lambda(\mu)$ denote the subspace of $L^2(\mu)$ generated by the eigenfunctions whose eigenvalue is $\lambda$. We also let $P_\lambda:L^2(\mu)\rightarrow L^2_\lambda(\mu)$ denote the orthogonal projection onto this subspace.
        \item[(iii)] A character $\lambda:V\rightarrow S^1$ is called a \emph{$\gn$-junta} if $\lambda(\gamma_n)=1$ for all but finitely many $n\in \mathbb{N}$. We let $\mathrm{J}_{\gn}$ denote the group of all $\gn$-juntas $\lambda:V\rightarrow S^1$ under pointwise multiplication.
        \item[(iv)] We also denote by $\mathrm{J}_{\gn}(\XX)$ the factor of $\XX$ generated by the eigenfunctions whose eigenvalue is a $\gn$-junta. In particular, $L^2(\mathrm{J}_{\gn}(\XX))=\bigoplus_{\lambda\in \mathrm{J}_{\gn}} L^2_\lambda(\mu)$.
    \end{itemize}
\end{definition}
We show that the Junta factor $\mathrm{J}_{\gn}(\XX)$ extends $Z^0_{\gn}(\XX)$.
\begin{lemma}[On $\gn$-juntas]\label{junta:properties}
    Let $V$ be a countable vector space over $\mathbb{F}_p$, let $\gn$ be a sequence in $V$ and let $\XX=(X,\mathcal{B},\mu,T)$ be a $V$-system. Then $\mathrm{J}_{\gn}$ is a subgroup of the Pontryagin dual $\Sigma$ of $V$, it contains $$\mathrm{Ann}(\gn) := \{\lambda\in \Sigma : \forall n\in\mathbb{N}\text{ }\lambda(\gamma_n)=1\}$$ and the quotient $\mathcal{J}_{\gn} := \mathrm{J}_{\gn}/\mathrm{Ann}(\gn)$ is at most countable.
\end{lemma}
\begin{proof}
    To prove the first claim, observe that if $\lambda_1,\lambda_2\in \mathrm{J}_{\gn}$, then we can find finite subsets $A_1,A_2\subseteq\mathbb{N}$ such that for all $i=1,2$ we have $\lambda_i(\gamma_n)=1$ for all $n\notin A_i$. In particular, we see that $(\lambda_1\cdot \lambda_2)(\gamma_n)=1$ for all $n\notin A_1\cup A_2$ and so $\lambda_1\cdot \lambda_2\in \mathrm{J}_{\gn}$. Since $\mathrm{J}_{\gn}$ is also closed under taking inverses, it is a subgroup of $\Sigma$. This subgroup contains $\mathrm{Ann}(\gn)$ by definition. Finally, we prove that $\mathcal{J}_{\gn}$ is at most countable. Let $V_{\gn}$ be the subspace of $V$ spanned by $\gn$. By Pontryagin duality, we can identify $\Sigma/\mathrm{Ann}(\gn)$ with the Pontryagin dual of $V_{\gn}$, which we denote by $\Sigma_{\gn}$. Thus, $\mathcal{J}_{\gn}\le \Sigma_{\gn}$ is just the subgroup of all $\gn$-juntas with respect to $V_{\gn}$. To show that this subgroup is at most countable, observe that the family of all finite subsets of $\mathbb{N}$ is countable. Given a fixed finite set $A\subseteq\mathbb{N}$, since $\gn$ generates $V_{\gn}$, there are only finitely many characters that are trivial on all $\gamma_n$ for $n\not \in A$. Thus, $\mathcal{J}_{\gn}$ is a countable union of finite sets and is therefore countable.
\end{proof}
In general, we will need to consider a mixed behavior between the two extremes illustrated in Example \ref{trivial:ex} and Example \ref{junta:ex}. This mixed behavior can be captured via the following definition.
\begin{definition}[Generalized eigenspace]
    Let $V$ be a countable vector space over $\mathbb{F}_p$, and let $\gn$ be a sequence of elemetns in $V$. For every $\lambda\in \mathcal{J}_{\gn}$, we let $\tilde{P}_{\lambda}:L^2(\mu)\rightarrow L^2(\mu)$ denote the orthogonal projection onto the subspace of $L^2(\mu)$ generated by all the eigenspaces $L^2_{\xi}(\mu)$ where $\xi \mod \mathrm{Ann}(\gn) = \lambda$.  
\end{definition}
The following theorem provides a limit formula for the ergodic averages along increasing $\IP$-F\o lner sequences. 
\begin{theorem}[Limit formula for the IP mean ergodic theorem.]\label{IPMETformula}
    Let $V$ be a countable vector space over $\mathbb{F}_p$, let $\XX=(X,\mathcal{B},\mu,T)$ be a $V$-system, and let $\gn$ be an arbitrary sequence of elements in $V$. Then for every increasing F\o lner sequence $\Phi=(\Phi_N)_{N\in\mathbb{N}}$, there exists a function $\omega_{\Phi}:\mathcal{J}_{\gn}\rightarrow \mathbb{C}$ with $|\omega|\le 1$ such that 
    \begin{equation}\label{IPMETlimitformula}\lim_{N\rightarrow\infty} \EIP T^\gamma f = \sum_{\lambda\in \mathcal{J}_{\gn}} \omega_{\Phi}(\lambda)\cdot \tilde{P}_\lambda f
    \end{equation} in $L^2(\mu)$, for all $f\in L^2(\mu)$. 
\end{theorem}
\begin{proof}
Let $V_{\gn}$ denote the subspace of $V$ generated by $\gn$. Choosing a basis for $V_{\gn}$ and extending this basis to $V$, we can find a complementary subspace $V'$ such that $V=V_{\gn}\oplus V'$. Let $\Sigma$ denote the Pontryagin dual of $V$ and $\Sigma_{\gn}$ the Pontryagin dual of $V_{\gn}$. From the decomposition $V=V_{\gn}\oplus V'$, we have $\Sigma = \Sigma_{\gn}\times \mathrm{Ann}(\gn)$. Equivalently, we can write every character $\xi:V\rightarrow S^1$ uniquely as a product of a character $\xi_{\gn} \in \Sigma_{\gn}$ and a character $\xi'\in \mathrm{Ann}(\gn)$. Throughout, we will identify $\Sigma_{\gn}$ as a subset of $\Sigma$ by setting $\xi_{\gn}(u\oplus u')=\xi(u)$ for $u\oplus u'\in V_{\gn}\oplus V'$.

We now follow the argument from the proof of Theorem \ref{IPMET}. In that proof, we showed that
$$\phi(\xi) = \lim_{N\rightarrow\infty}\sEIP \xi(\gamma_n)$$ exists pointwise. Furthermore, it is easy to see that $\phi$ is a homomorphism in $\xi$, and thus it suffices to compute $\phi(\xi_{\gn})$ and $\phi(\xi')$ separately. Since $\xi'$ annihilates $V_{\gn}$, a direct computation shows that $\phi(\xi')=1$. Now, for $\phi(\xi_{\gn})$ we have two cases. If $\xi_{\gn}\notin \mathcal{J}_{\gn}$, then $\phi(\xi_{\gn})=0$ (see \eqref{avgtoprod}). Otherwise, we denote by $\omega_{\Phi}$ the restriction of $\phi$ to $\mathrm{J}_{\gn}$. Since $\phi(\xi')=1$, this map factors through the quotient $\mathcal{J}_{\gn}$. Therefore, by the spectral theorem, and Lemma \ref{junta:properties}, 
\begin{align*}\lim_{N\rightarrow\infty}\sEIP T^\gamma f &= \int_{\Sigma} \phi(\xi)~d\rho(\xi)\\ &= \int_{\Sigma_{\gn}\times \mathrm{Ann}(\gn)} \phi(\xi_{\gn})\phi(\xi')\,d\rho(\xi)
\\&=\int_{\mathcal{J}_{\gn}\times\mathrm{Ann}(\gn)}\omega_{\Phi}(\xi_{\gn})\,d\rho(\xi)\cdot \\&=\sum_{\lambda\in \mathcal{J}_{\gn}} \omega_{\Phi}(\lambda)\cdot \tilde{P}_{\lambda}.
\end{align*}
To justify the last equality, let $\Sigma_{\lambda} =\{\lambda\}\times \mathrm{Ann}(\gn)$. We need to show that $\int_{\Sigma} 1_{\Sigma_{\lambda}} \, d\rho(\xi) = \rho( \Sigma_\lambda)$ is equal to the projection $\tilde{P}_{\lambda}$. By the spectral theorem, the image of the projection $\rho(\Sigma_{\lambda})$ is the closed subspace generated by the eigenspaces $L^2_\xi(\mu)$ for all $\xi \in \Sigma_\lambda$. Since $\Sigma_\lambda$ consists precisely of those characters $\xi$ such that $\xi \pmod{\mathrm{Ann}(\gn)} = \lambda$, this is exactly the definition of the projection $\tilde{P}_\lambda$. Thus, $\rho(\Sigma_\lambda) = \tilde{P}_\lambda$, which completes the proof.
\end{proof}
The next proposition is inspired by a key observation. Namely, by removing finitely many elements from each set in the F\o lner sequence $\Phi=(\Phi_N)_{N\in\mathbb{N}}$ we can ensure that the limit in \eqref{IPMETlimitformula} gets arbitrarily close to an actual projection. In fact, it gets arbitrarily close to the projection of $f$ onto the junta factor. Before we formalize this, we give an example that best illustrates this phenomenon. Consider the following modification of Example \ref{junta:ex}.
\begin{example}[Example \ref{junta:ex} revised]
Let $X = \mathbb{F}_p^\mathbb{N}$, let $V=\mathbb{F}_p^\omega$ and equip $X$ with the usual action $T^\gamma x = x+\gamma$. Again, let $\gn=(e_n)_{n\in \mathbb{N}}$ denote the standard basis for $V$, and let $f:X\rightarrow\mathbb{C}$ denote the projection onto the first coordinate. Now, consider the F\o lner sequence $\Phi'=([2,N])_{N\in \mathbb{N}}$. In that case, since $f$ depends only on the first coordinate of $\XX$, we end up in a situation which resembles that of Example \ref{trivial:ex}. Formally, we have that
    $\E_{\gamma\in \IP_{[2,N]}(\gn)}T^\gamma f = f$ for all $N$. 
\end{example}
More generally, we have the following result.
\begin{proposition}[A double-limit mean ergodic theorem along IPs.]\label{refined}
    Let $V$ be a countable vector space over $\mathbb{F}_p$ and let $\gn$ be an arbitrary sequence of elements in $V$. Then for every $V$-system $\XX=(X,\mathcal{B},\mu,T)$, every function $f\in L^2(\mu)$, and every increasing F\o lner sequence $\Phi=(\Phi_N)_{N\in\mathbb{N}}$ we have
    $$\lim_{D\rightarrow\infty} \lim_{N\rightarrow\infty} \E_{\gamma\in \IP_{\Phi_N^D}(\gn)} T^\gamma f = \E(f\mid \mathrm{J}_{\gn}(\XX)),$$
    in $L^2(\mu)$, where $\Phi_N^D = \Phi_N\setminus \{1,2,\dots,D\}$.
\end{proposition}
\begin{proof}[Proof of Proposition \ref{refined}]
    For every $D\in\mathbb{N}$, write $\Phi^D = (\Phi_N^D)_{N\in\mathbb{N}}$. By Theorem \ref{IPMETformula}, we have
    $$\lim_{N\rightarrow\infty} \E_{\gamma\in \IP_{\Phi_N^D}(\gn)} T^\gamma f = \sum_{\lambda\in\mathcal{J}_{\gn}} \omega_{\Phi^D}(\lambda)\cdot \tilde{P}_{\lambda} f.$$
    Since the images of $\tilde{P}_{\lambda}$ are orthogonal and $|\omega_{\Phi^D}(\lambda)|\le 1$, we have the uniform bound \begin{equation}\label{finitenorm}\norm{\sum_{\lambda\in\mathrm{J}_{\gn}} \omega_{\Phi^D}(\lambda)\cdot \tilde{P}_{\lambda} (f)}_{L^2(\mu)}\le \norm{f}_{L^2(\mu)}<\infty
    \end{equation}
    for all values of $D\in \mathbb{N}$. 
    Moreover, we have (see \eqref{avgtoprod})
    $$\omega_{\Phi^D}(\lambda) = \lim_{N\rightarrow\infty}\prod_{n\in \Phi_N^D} \left(\frac{1+\lambda(\gamma_n)}{2}\right).$$
  We see that for every $\lambda\in \mathcal{J}_{\gn}$, there exists $D$ sufficiently large such that $\omega_{\Phi^D}(\lambda)=1$ (since $\lambda(
  \gamma_n)=1$ for all but finitely many $n$). In other words, $\lim_{D\rightarrow\infty}\omega_{\Phi^D}(\lambda)=1$ pointwise. Thus, by the dominated convergence theorem, we see from \eqref{finitenorm} that
    $$\lim_{N\rightarrow\infty} \E_{\gamma\in \IP_{\Phi_N^D}(\gn)} T^\gamma f=\sum_{\lambda\in \mathrm{J}_{\gn}} \tilde{P}_{\lambda}(f) = \E(f\mid\mathrm{J}_{\gn}(\XX)).$$ 
\end{proof}
We will now illustrate our methods by proving Lemma \ref{simple}.
\begin{proof}[Proof of Lemma \ref{simple}]
The lemma is trivial when $k=1$. When $k=2$, we see by the Furstenberg correspondence principle that it suffices to show that $c_{\IP}^{\rec}(2,\delta,\mathbb{F}_p^\omega) \ge \delta^2 $ and $c(2,\delta,\mathbb{F}_p^\omega)\le \delta^2$. To prove the first inequality, we rely on Proposition \ref{refined}. Let $V=\mathbb{F}_p^\omega$, let $\gn$ be a sequence of elements in $V$, and let $A$ be an arbitrary measurable subset in some $V$-system $\XX=(X,\mathcal{B},\mu,T)$ with $\mu(A)\ge \delta$. We want to show that for every $\varepsilon>0$, the set 
$$\{\gamma\in \IP(\gn)\setminus\{0\} : \mu(A\cap T^{-\gamma}A) > \delta^2-\varepsilon\}$$ is non-empty.

Assume for the sake of contradiction that this is false. Then for some $\varepsilon_0>0$ we have
\begin{equation}\label{contradiction:simple}
    \mu(A\cap T^{-\gamma}A) \le \delta^2-\varepsilon_0
\end{equation}
for all $\gamma\in \IP(\gn)\setminus\{0\}$. Take the F\o lner sequence $\Phi=([1,N])_{N\in\mathbb{N}}$, and as usual, let $\Phi^D = ([D,N])_{N\in\mathbb{N}}$. From Proposition \ref{refined}, applied to $f=1_A$, we know that by taking the limit as $N\rightarrow\infty$ and then $D\rightarrow \infty$, the ergodic averages of $T^{\gamma}1_A$ converges to the projection onto the junta factor. Thus, we can find arbitrarily large $D$ and $N$ such that
$$\E_{\gamma\in \IP_{\Phi_N^D}(\gn)} \mu(A\cap T^{-\gamma}A) \ge \int_X 1_A \E(1_A\mid \mathrm{J}_{\gn}(\XX))\,d\mu - \varepsilon_0.$$
By Jensen's inequality, $$\int_X 1_A\cdot  \E(1_A\mid \mathrm{J}_{\gn}(\XX))\,d\mu = \int_X \E(1_A\mid \mathrm{J}_{\gn}(\XX))^2\,d\mu \ge \mu(A)^2\ge \delta^2.$$ Combining this with the inequality above, we obtain $$\E_{\gamma\in \IP_{\Phi_N^D}(\gn)} \mu(A\cap T^{-\gamma}A) \ge \delta^2 - \frac{\varepsilon_0}{2},$$ which is a contradiction to \eqref{contradiction:simple}. 
Thus, $c_{\IP}^{\rec}(2,\delta,\mathbb{F}_p^\omega)\ge \delta^2$. The inequality $c(2,\delta,\mathbb{F}_p^\omega)\le \delta^2$ follows from Lemma \ref{randombound}.
\end{proof}
\section{Characteristic factors for multiple ergodic averages along IPs}\label{characteristicfactors:sec}
In this section, we study multiple ergodic averages along IPs.
\begin{definition}[Multiple ergodic averages along IPs]
    Let $V$ be a countable vector space over $\mathbb{F}_p$, let $\XX=(X,\mathcal{B},\mu,T)$ be a $V$-system, let $f_1,\dots,f_k\in L^\infty(\mu)$ be $k$ bounded functions on $\XX$, for some $k\ge 1$, let $\gn$ be a sequence of elements in $V$, and let $\Phi=(\Phi_N)_{N\in\mathbb{N}}$ be a F\o lner sequence. The \emph{multiple ergodic average of $f_1,\dots,f_k$ along the $\IP$-F\o lner sequence $\IP_{\Phi_N}\left(\gn\right)$} is the sequence
    \begin{equation}\label{mea}
        A_N(f_1,\dots,f_k)=\EIP \prod_{i=1}^k T^{i\gamma} f_i.
    \end{equation}
\end{definition}
Our next goal is to prove the existence and establish a limit formula for these averages.

We begin with the following version of the van der Corput lemma for IPs (see \cite[Lemma 4.3]{SK}).
\begin{lemma}[van der Corput lemma for increasing $\IP$-F\o lner sequences]\label{vdc}
    Let $H$ be a Hilbert space with norm $\|\cdot\|$ and inner product $\left<\cdot,\cdot \right>$, and let $(x_n)_{n\in \mathbb{N}}$ be a sequence in $H$. Let $V$ be a vector space over $\mathbb{F}_p$, let $\gn$ be a sequence of elements in $V$, and let $\Phi=(\Phi_N)_{N\in\mathbb{N}}$ be an increasing F\o lner sequence. Then,
    $$\limsup_{N\rightarrow\infty} \|\EIP x_\gamma \|^2 \le \limsup_{M\rightarrow\infty} \E_{\gamma_1,\gamma_2\in \IP_{\Phi_M}\left(\gn\right)}\limsup_{N\rightarrow\infty} \E_{\gamma\in \IP_{\Phi_N^M}\left(\gn\right)} \left<x_{\gamma+\gamma_1},x_{\gamma+\gamma_2}\right>.$$
\end{lemma}
\begin{proof}
    The proof is precisely as in \cite{SK}. We provide the details here for the sake of completeness. Let $\beta(\gamma,\gamma') = \gamma+\gamma'$, and observe that $\beta$ induces a bijection $\beta:\IP_{\Phi_M}\left(\gn\right)\times\IP_{\Phi_N^M}\left(\gn\right)\rightarrow \IP_{\Phi_N}\left(\gn\right),$ for all $N,M\in \mathbb{N}$, provided that $N>M$. In that case we have
    $$\EIP x_\gamma = \E_{\gamma\in \IP_{\Phi_N^M\left(\gn\right)} \E_{\gamma'\in \IP_{\Phi_M}\left(\gn\right)}\left(\gn\right)} x_{\gamma+\gamma'}.$$
    From Jensen's inequality, we deduce that
    \begin{align*}\norm{\E_{\gamma\in \IP_{\Phi_N^M\left(\gn\right)}} \E_{\gamma'\in \IP_{\Phi_M}\left(\gn\right)} x_{\gamma+\gamma'}}^2 &\le\E_{\gamma\in \IP_{\Phi_N^M\left(\gn\right)}} \norm{\E_{\gamma'\in \IP_{\Phi_M}\left(\gn\right)} x_{\gamma+\gamma'}}^2\\&= \E_{\gamma\in \IP_{\Phi_N^M\left(\gn\right)}} \E_{\gamma_1,\gamma_2\in \IP_{\Phi_M}\left(\gn\right)}\langle x_{\gamma+\gamma_1},x_{\gamma+\gamma_2}\rangle.\end{align*}
    Interchanging summations and then taking the limit as $N\rightarrow\infty$ on both sides, we get
     $$\limsup_{N\rightarrow\infty} \norm{\EIP x_\gamma }^2 \le  \E_{\gamma_1,\gamma_2\in \IP_{\Phi_M}\left(\gn\right)}\limsup_{N\rightarrow\infty} \E_{\gamma\in \IP_{\Phi_N^M}\left(\gn\right)} \left<x_{\gamma+\gamma_1},x_{\gamma+\gamma_2}\right>$$
     for all $M\in\mathbb{N}$. Now taking the limit as $M\rightarrow\infty$, the claim follows.
\end{proof}
The idea now is to follow the argument from the work of Host and Kra \cite{host2005nonconventional}. First, we must define $\IP$ counterparts for the Host--Kra seminorms and the Host--Kra cubic measures for vector spaces over finite fields. We start with the former.
\begin{definition}[IP Host--Kra seminorms over finite fields]
   Let $V$ be a countable vector space over $\mathbb{F}_p$, and let $\gn\subseteq V$ be an arbitrary sequence. Let $\XX = (X,\mathcal{B},\mu,T)$ be a $V$-system. For $k = 1$, define the first seminorm $\|\cdot\|_{U^1_{\IP}(\XX,\gn)}$ by setting
   $$\|f\|_{U^1_\IP (\XX,\gn)} := \sup_{\Phi} \max_{1\le i \le p-1} \lim_{N\rightarrow\infty} \norm{\E_{\IP_{\Phi_N}(\ign)} T^\gamma f}_{L^2(\mu)},$$ for all $f\in L^\infty(\mu)$,
   where the supremum is taken over all increasing F\o lner sequences $\Phi=(\Phi_N)_{N\in \mathbb{N}}$. For $k\ge 1$, define the $(k+1)^{\mathrm{st}}$ seminorm $\|\cdot\|_{U^{k+1}_{\IP}(\XX,V)}$ by setting 
   $$\|f\|_{U^{k+1}_{\IP}(\XX,\gn)} := \sup_{\Phi} \max_{1\le i\le p-1}\lim_{N\rightarrow\infty}  \left(\E_{\gamma_1,\gamma_2 \in \IP_{\Phi_N}(\ign))} \|T^{\gamma_1}f\cdot \overline{T^{\gamma_2} f}\|_{U^k_{\IP}(\XX,\gn)}^{2^k}\right)^{1/2^{k+1}},$$
   again taking the supremum over all increasing F\o lner sequences $\Phi=(\Phi_N)_{N\in \mathbb{N}}$.
\end{definition}

As in the case of the Host--Kra seminorms, these $\IP$-seminorms are characterized in terms of certain cubic measures, inductively constructed analogously to how the cubic measures are constructed in \cite{host2005nonconventional}.  Let $[k]=\{0,1\}^k$ be the $k$-dimensional cube. Using notation of \cite{host2005nonconventional}, we write a vertex $\epsilon\in [k]$ without commas and set  $|\epsilon|=\sum_{i=1}^k \epsilon_i$.  If  $\epsilon\in [k]$ and $\eta\in [\ell]$, we  concatenate them to obtain an element $\epsilon\eta\in[k+\ell]$.  Let $\mathcal{C}\colon \mathbb{C}\rightarrow\mathbb{C}$ denote the complex conjugation.  
In~\cite{host2005nonconventional}, joinings are built inductively over the invariant factor, and in our inductive construction the invariant factor is replaced by the $\gn$-junta factor.

\begin{definition}[IP cubic measures over finite fields]
    Let $V$ be a countable vector space over $\mathbb{F}_p$. For every $V$-system $\XX = (X,\mathcal{B},\mu,T)$, set $X^{[0]}=X$ and $\mu_{\gn}^{[0]} = \mu$. For $k\ge 1$, identifying $X^{[k+1]}$ with $X^{[k]}\times X^{[k]}$, we define the cubic measure $\mu_{\gn}^{[k+1]}$ on $X^{[k+1]}$ to be the relatively independent joining of $\XX^{[k]}$ with itself over the factor $\mathrm{J}_{\gn}(\XX^{[k]})$.  This means that if $(f_\epsilon)_{\epsilon\in [k+1]}$ are $2^{k+1}$ bounded real-valued functions on $X$, the measure $\mu_{\gn}^{[k+1]}$ is defined by
    \begin{align*}\int_{X^{[k+1]}}& \bigotimes_{\epsilon\in [k+1]} \tilde{f}_\epsilon \,d\mu_{\gn}^{[k+1]} =\\\int_{X^{[k]}} \E&\left(\bigotimes_{\eta\in [k]} f_{\eta0}\bigm| \mathrm{J}_{\gn}(\XX^{[k]})\right)\cdot \E\left(\bigotimes_{\eta\in [k]} f_{\eta1}\bigm| \mathrm{J}_{\gn}(\XX^{[k]})\right)\,d\mu_{\gn}^{[k]}.
    \end{align*}
    The action on $X^{[k+1]}$, $T^{[k+1]}$ is defined as the diagonal action $(T^{[k+1]})^\gamma = (T^{[k]})^\gamma\times (T^{[k]})^\gamma$.
\end{definition}
The following simple lemma will be useful soon.
\begin{lemma}\label{muign}
Let $V$ be a countable vector space over $\mathbb{F}_p$, let $\XX$ be a $V$-system, let $\gn$ be a sequence of elements in $V$ and let $i\in \mathbb{F}_p\setminus\{0\}$. Then for every $k\ge 0$ we have $$\mu^{[k]}_{\gn} = \mu^{[k]}_{\ign}.$$
\end{lemma}
\begin{proof}
    By definition, $\mathrm{J}_{\gn} = \mathrm{J}_{\ign}$. Indeed, a character of $V$ evaluates to $1$ on $\gamma_n$ for all but finitely many $n$ if and only if it evaluates $1$ on $i\cdot \gamma_n$ for the same set of $n$. It then follows immediately that $\mathrm{J}_{\gn}(\XX) = \mathrm{J}_{\ign}(\XX)$ and so a simple proof by induction on $k$ gives the desired claim.
\end{proof}

The following Proposition is a variant of \cite[Proposition 4.7]{SK}.
\begin{proposition}[IP-seminorms are controlled by the IP cubic measures]
    Let $V$ be a countable vector space over $\mathbb{F}_p$. Let $\XX=(X,\mathcal{B},\mu,T)$ be a $V$-system, let $\gn$ be a sequence of elements in $V$, and let $k\ge 1$. Then for every $f\in L^\infty(\mu)$, we have
    $$\|f\|_{U^k_{\IP}(\XX,(\gamma_n)_{n\in\mathbb{N}})}^{2^k} \le \int_{X^{[k]}} \bigotimes_{\epsilon\in [k]} \mathcal{C}^{|\epsilon|} f \,d\mu_{\gn}^{[k]}. $$ Here $\mathcal{C}$ is the complex-conjugation map $z\mapsto \overline{z}$.
\end{proposition}
\begin{proof}
We prove the proposition by induction on $k$. When $k=1$, Theorem \ref{IPMETformula} gives
\begin{align*}
    \norm{f}_{U^1_{\IP}(\XX,\gn)}^2 &= \sup_{\Phi}\max_{1\le i \le p-1} \lim_{N\rightarrow\infty} \norm{\E_{\gamma\in \IP_{\Phi_N}(\ign)} T^\gamma f}_{L^2(\mu)}^2\\&=\sup_{\Phi}\max_{1\le i \le p-1}\norm{\sum_{\lambda\in \mathcal{J}_{\ign}}\omega^{(i)}_{\Phi}(\lambda)\cdot \tilde{P}_{\lambda}(f)}_{L^2(\mu)}^2
\end{align*} 
Recall that $\mathrm{J}_{\gn} = \mathrm{J}_{\ign}$. Since $|\omega^{(i)}_{\Phi}|\le1$ and generalized eigenspaces associated with different eigenvalues are orthogonal, we deduce that the above is bounded by
\begin{align*}\norm{\sum_{\lambda\in \mathcal{J}_{\ign}} \tilde{P}_{\lambda}(f)}_{L^2(\mu)}^2 &= \norm{\E(f\mid \mathrm{J}_{\ign}(\XX))}_{L^2(\mu)}^2\\&= \norm{\E(f\mid \mathrm{J}_{\gn}(\XX))}_{L^2(\mu)}^2\\& = \int_{X^{[1]}} f\otimes\overline{f}\,d\mu_{\gn}^{[1]}.
\end{align*}
Let $k\ge 2$. From Theorem \ref{IPMETformula}, Lemma \ref{muign}, the induction hypothesis, and the bound $|\omega^{(i)}_{\Phi}|\le 1$, we have 
\begin{align*}
    \norm{f}&_{U^k_{\IP}(\XX,\gn)}^{2^k} \\&= \sup_{\Phi} \max_{1\le i \le p-1}\lim_{N\rightarrow\infty} \E_{\gamma_1,\gamma_2\in \IP_{\Phi_N}(\ign)} \|T^{\gamma_1}f\cdot \overline{T^{\gamma_2}f}\|_{U^{k-1}_{\IP}(\XX,\ign)}^{2^{k-1}}\\&\le \sup_{\Phi} \max_{1\le i \le p-1}\lim_{N\rightarrow\infty}\E_{\gamma_1,\gamma_2\in \IP_{\Phi_N}(\ign)} \int_{X^{[k-1]}} \bigotimes_{\epsilon\in [k-1]}\mathcal{C}^{|\epsilon|}\left(T^{\gamma_1}f\cdot\overline{T^{\gamma_2}f}\right)\,d\mu_{\ign}^{[k-1]}\\&=\sup_{\Phi} \max_{1\le i \le p-1}\lim_{N\rightarrow\infty} \int_{X^{[k-1]}} \left|\E_{\gamma\in \IP_{\Phi_N}(\ign)} \bigotimes_{\epsilon\in [k-1]}\mathcal{C}^{|\epsilon|}T^\gamma f\right|^2\,d\mu_{\ign}^{[k-1]}
    \\&=\sup_{\Phi} \max_{1\le i \le p-1}\lim_{N\rightarrow\infty} \int_{X^{[k-1]}} \left|\E_{\gamma\in \IP_{\Phi_N}(\ign)} (T^{[k-1]})^\gamma \bigotimes_{\epsilon\in [k-1]}\mathcal{C}^{|\epsilon|} f\right|^2\,d\mu_{\ign}^{[k-1]}\\&\le  \max_{1\le i \le p-1}\int_{X^{[k-1]}}\E\left(\bigotimes_{\epsilon\in [k-1]}\mathcal{C}^{|\epsilon|}f\mid \mathrm{J}_{\ign}(\XX^{[k-1]})\right) \cdot \overline{\E\left(\bigotimes_{\epsilon\in [k-1]}\mathcal{C}^{|\epsilon|}f\mid \mathrm{J}_{\ign}(\XX^{[k-1]})\right)}\,d\mu_{\ign}^{[k-1]}\\&=\int_{X^{[k-1]}}\E\left(\bigotimes_{\epsilon\in [k-1]}\mathcal{C}^{|\epsilon|}f\mid \mathrm{J}_{\gn}(\XX^{[k-1]})\right) \cdot \overline{\E\left(\bigotimes_{\epsilon\in [k-1]}\mathcal{C}^{|\epsilon|}f\mid \mathrm{J}_{\gn}(\XX^{[k-1]})\right)}\,d\mu_{\gn}^{[k-1]}
    \\&=
    \int_{X^{[k]}} \bigotimes_{\epsilon\in [k]} \mathcal{C}^{|\epsilon|} f \,d\mu_{\gn}^{[k]}.
\end{align*}
completing the argument.
\end{proof}
Next, we show that the IP seminorms control multiple ergodic averages.
\begin{proposition}\label{seminormcontrol:prop}
    Let $V$ be a vector space over $\mathbb{F}_p$, let $\gn$ be a sequence of elements in $V$, let $\XX=(X,\mathcal{B},\mu,T)$ be a $V$-system, let $1\le k< p$, and let $f_1,\dots,f_k\in L^\infty(\mu)$ be bounded functions satisfying that  $\|f_i\|_{L^\infty(\mu)}\le 1$, for all $1\le i \le k$. Then for any increasing F\o lner sequence $\Phi=(\Phi_N)_{N\in\mathbb{N}}$, we have
    \begin{equation}\label{seminormcontrol}
    \limsup_{N\rightarrow\infty} \norm{\EIP \prod_{i=1}^k T^{i\gamma }f_i}_{L^2(\mu)} \le \min_{1\le j \le k}\|f_j\|_{U^k_{\IP}(\XX,\gn)}. 
    \end{equation}
\end{proposition}
\begin{proof}
    We prove this proposition by induction on $k$. When $k=1$, the claim follows directly from the definition. Let $k\ge 2$, and assume inductively that the claim holds for all smaller values of $k$. Setting $x_\gamma  = \prod_{i=1}^k T^{i\gamma} f_i$, Lemma \ref{vdc} gives that 
    \begin{align*}
        &\limsup_{N\rightarrow\infty} \norm{\EIP \prod_{i=1}^k T^{i\gamma }f_i}_{L^2(\mu)}^2\\&\le \limsup_{M\rightarrow\infty}\E_{\gamma_1,\gamma_2 \in \IP_{\Phi_M}} \limsup_{N\rightarrow\infty}\E_{\gamma\in \IP_{\Phi_N^M}} \int_X \prod_{i=1}^k T^{i(\gamma+\gamma_1)}f_i\cdot \overline{T^{i(\gamma+\gamma_2)}f_i}\,d\mu\\&=\limsup_{M\rightarrow\infty}\E_{\gamma_1,\gamma_2 \in \IP_{\Phi_M}} \limsup_{N\rightarrow\infty}\E_{\gamma\in \IP_{\Phi_N^M}} \int_X \left(T^{\gamma_1} f_1\cdot \overline{T^{\gamma_2}f_1}\right)\cdot \left(\prod_{i=1}^{k-1}T^{i\gamma}\left(T^{(i+1)\gamma_1}f_{i+1} \cdot \overline{T^{(i+1)\gamma_2}f_{i+1}}\right)\right)\,d\mu,
    \end{align*}
    where in the last equality we used the fact that $T^\gamma$ preserves the measure $\mu$.
    By the Cauchy-Schwarz inequality, and the induction hypothesis, we can bound the above by
\begin{align*}\limsup_{M\rightarrow\infty} &\E_{\gamma_1,\gamma_2\in \IP_{\Phi_M}} \norm{T^{\gamma_1}f_1\cdot \overline{T^{\gamma_2}f_1}}_{L^2(\mu)}\cdot \norm{ \limsup_{N\rightarrow\infty}\E_{\gamma\in \IP_{\Phi_N^M}}\left(\prod_{i=1}^{k-1}T^{i\gamma}\left(T^{(i+1)\gamma_1}f_{i+1} \cdot \overline{T^{(i+1)\gamma_2}f_{i+1}}\right)\right)}_{L^2(\mu)} \\&\le \limsup_{M\rightarrow\infty}\E_{\gamma_1,\gamma_2\in \IP_{\Phi_M}} \norm{\limsup_{N\rightarrow\infty} \E_{\gamma\in \IP_{\Phi_N}}\left(\prod_{i=1}^{k-1}T^{i\gamma}\left(T^{(i+1)\gamma_1}f_{i+1} \cdot \overline{T^{(i+1)\gamma_2}f_{i+1}}\right)\right)}_{L^2(\mu)} \\&\le \min_{1\le i \le k-1}\limsup_{M\rightarrow\infty} \E_{\gamma_1,\gamma_2\in \IP_{\Phi_M}} \norm{T^{(i+1)\gamma_1}f_{i+1}\cdot \overline{T^{(i+1)\gamma_2}f_{i+1}}}_{U^{k-1}_{\IP}(\XX,(\gn))}\\&=\min_{2\le i \le k} \|f_i\|_{U^{k}_{\IP}(\XX,\gn)}.
\end{align*}
To prove that it is also smaller than $\|f_1\|_{U^k_{\IP}(\XX,\gn)}$, one may replace $T$ with $T^{-1}$, since $T$ preserves the $L^2(\mu)$-norm, one may apply the same argument with $f_{k-i+1}$ taking the role of $f_i$, in particular, $f_1$ now takes the role of $f_k$.
\end{proof}
Finally, we will show that while the $\IP$-seminorms are generally larger than the Host--Kra seminorms, their characteristic factors are equivalent for $k>1$ (note that Theorem \ref{IPMETformula} shows that they are not equivalent for $k=1$). The analogue of the Host--Kra factors for vector spaces over finite fields were defined and studied in \cite{btz,btz2}. Throughout, we let $\|\cdot\|_{U^k(\XX,V)}$ denote the $k^{\mathrm{th}}$ Host-Kra seminorm, and $Z^k(\XX,V)$ denote the $k^{\mathrm{th}}$ Host--Kra factor, associated with the action on $V$ on the measure space $\XX=(X,\mathcal{B},\mu)$.
\begin{lemma}\label{measurecontrol}
Let $V$ be a countable vector space over $\mathbb{F}_p$, let $\gn$ be a sequence of elements in $V$, let $\XX=(X,\mathcal{B},\mu,T)$ be a $V$-system, let $k\ge 2$, and let $(f_\epsilon)_{\epsilon\in [k]}$ be $2^k$ bounded functions on $\XX$. If $\|f_\epsilon\|_{U^k(\XX,V_{\gn})} = 0$ for some $\epsilon\in [k]$, where $V_{\gn}$ is the subspace of $V$ generated by $\gn$, then
$$\int_X \bigotimes_{\epsilon\in [k]} \mathcal{C}^{|\epsilon|} f_\epsilon \,d\mu_{\gn}^{[k]} = 0.$$
\end{lemma}
\begin{proof}
    For every $\ell\in\mathbb{N}$, let $\mathrm{J}^{(\ell)}_{\gn}\subseteq \mathrm{J}_{\gn}$ be the subgroup generated by all the characters $\xi$ satisfying that $\xi(\gamma_n)=1$ for all $n\ge \ell$. It is easy to see that $\mathrm{J}^{(\ell)}_{\gn}$ are increasing subgroups and $\mathrm{J}_{\gn}=\bigcup_{\ell=1}^\infty\mathrm{J}^{(\ell)}_{\gn}$ is the direct limit. In particular, for each $\ell$, and every $V$-system $\XX$, we obtain a factor $\mathrm{J}^{(\ell)}_{\gn}(\XX)$ generated by the eigenfunctions whose eigenvalues are in $\mathrm{J}^{(\ell)}_{\gn}$. From the construction, \begin{equation}\label{juntaapprox}\mathrm{J}_{\gn}(\XX)=\bigvee_{\ell\ge 1} \mathrm{J}^{(\ell)}_{\gn}(\XX)
    \end{equation}
    is the minimal factor extending  $\mathrm{J}^{(\ell)}_{\gn}(\XX)$ for all $\ell\ge 1$ (i.e., the inverse limit). For each $\ell\ge 1$, we may define $\XX_{\ell}^{[k]}$ and $\mu_{\ell,\gn}^{[k]}$ the same way as we define $\XX^{[k]}$ and $\mu_{\gn}^{[k]}$, but replacing every instance of $\mathrm{J}_{\gn}(\XX)$ with $\mathrm{J}_{\gn}^{(\ell)}(\XX)$.\\
   \noindent \textbf{Claim:} For every $k\ge 0$, we have
   \begin{itemize}
    \item[(1)] $\lim_{\ell\rightarrow\infty} \mu_{\ell,\gn}^{[k]} = \mu_{\gn}^{[k]}$ in the weak$^*$ topology.
    \item[(2)] $\XX^{[k]}$ is the inverse limit of the increasing sequence of factors $\XX_{\ell}^{[k]}$ as $\ell\rightarrow\infty$.
   \end{itemize}
   \noindent \textbf{Proof of claim:} 
   We prove these two statements simultaneously by induction on $k$. When $k=0$, we have $\XX_\ell^{[0]} = \XX^{[0]}$, and $\mu_{\ell,\gn}^{[0]} = \mu_{\gn}^{[0]}$ for all $\ell\ge 1$, hence the claim follows trivially. Let $k\ge 1$, and assume inductively that we have already established $(1)$ and $(2)$ for smaller values of $k$. We prove $(1)$ for this value of $k$ first. Let $f,g$ be continuous functions on $X^{[k-1]}$ and write
   \begin{align*}
       \tilde{f} = \E_{\mu^{[k-1]}_{\gn}}\left(f\mid \mathrm{J}_{\gn}(\XX^{[k-1]})\right) \intertext{and}
       \tilde{g} = \E_{\mu^{[k-1]}_{\gn}}\left(g\mid \mathrm{J}_{\gn}(\XX^{[k-1]})\right).
   \end{align*}
  Now, observe that $\mathrm{J}_{\gn}(\XX^{[k-1]})$ is the inverse limit of $\mathrm{J}_{\gn}^{(\ell)} (\XX^{[k-1]}_{\ell})$. Indeed, take $h$ measurable with respect to $\mathrm{J}_{\gn}(\XX^{[k-1]})$, then for all $\ell_1$ sufficiently large, it is arbitrarily close to a function in $\mathrm{J}_{\gn}^{(\ell_1)}(\XX^{[k-1]})$, which can then be arbitrarily approximated by $\mathrm{J}_{\gn}^{(\ell_1)}(\XX_{\ell_2}^{[k-1]})$ for all $\ell_2$ sufficiently large. Since all the factors are increasing, the latter is smaller than
   $\mathrm{J}_{\gn}^{(\ell)}(\XX_{\ell}^{[k-1]})$ where $\ell=\max(\ell_1,\ell_2)$. Hence, by the Doob's convergence theorem, we have that $\tilde{f}_{\ell}\rightarrow \tilde{f}$ and $\tilde{g}_{\ell}\rightarrow \tilde{g}$ pointwise $\mu^{[k-1]}_{\gn}$-almost everywhere, where
     \begin{align*}
       \tilde{f}_{\ell} = \E_{\mu^{[k-1]}_{\ell,\gn}}\left(f\mid \mathrm{J}^{(\ell)}_{\gn}(\XX_{\ell}^{[k-1]}) \right)\intertext{and}
       \tilde{g}_{\ell} = \E_{\mu^{[k-1]}_{\ell,\gn}}\left(g\mid \mathrm{J}^{(\ell)}_{\gn}(\XX_{\ell}^{[k-1]})\right) .
   \end{align*}
    By Egorov's theorem, for every $\varepsilon>0$, we can identify a set $A\subseteq X^{[k-1]}$ with $\mu_{\gn}^{[k-1]}(A)<\varepsilon$, such that the above pointwise convergences are uniform outside of $A$. Thus we have
   \begin{align*}\int_{X^{[k]}} f\otimes g\,d\mu_{\gn}^{[k]} &= \int_{X^{[k-1]}} \tilde{f}\cdot \tilde{g} \,d\mu_{\gn}^{[k-1]}\\&\approx_{\varepsilon} \int_{X^{[k-1]}\setminus A} \tilde{f}\cdot \tilde{g}\,d\mu_{\gn}^{[k-1]}\\
   &\approx_{\varepsilon} \int_{X^{[k-1]}\setminus A} \tilde{f}_{\ell}\cdot \tilde{g}_{\ell}\,d\mu_{\ell,\gn}^{[k-1]}\\
   &\approx_{\varepsilon}\int_{X^{[k-1]}} \tilde{f}_{\ell}\cdot \tilde{g}_{\ell}\,d\mu_{\ell,\gn}^{[k-1]}\\&= \int_{X^{[k]}} f\otimes g\,d\mu_{\ell,\gn}^{[k]}
   \end{align*}
   where by $\approx_\varepsilon$ we mean that the difference between the two components is  $o_{\varepsilon\rightarrow0}(1)$, and here $\ell$ is sufficiently large with respect to $\varepsilon$. This proves $(1)$. Property $(2)$ follows from $(1)$, the induction hypothesis and \eqref{juntaapprox}, completing the proof of the claim.

   We return to the proof of the lemma. For every $\ell\ge 1$, let $V_{\ell}$ denote the subspace of $V_{\gn}$ generated by all the elements of $\gn$ other than the first $\ell$ elements. Clearly $V_{\ell}$ has finite index in $V_{\gn}$. Here we need to apply an important result of Leibman, originally established for $\mathbb{Z}$-actions in \cite{leibman-hkz} (see \cite[Lemma A.1]{ABS} for general countable abelian group actions), which asserts that $\|f_{\epsilon}\|_{U^k(\XX,V_{\gn})} = 0$ if and only if $\|f_{\epsilon}\|_{U^k(\XX,V_{\ell})}=0$ for all $\ell\ge 1$. Our main observation now is the fact that $\mathrm{J}^{(\ell)}_{\gn}(\XX)$ is precisely the $V_{\ell}$-invariant factor of $\XX$, and therefore it follows from the work of Host and Kra that 
   $$\int_{\XX} \bigotimes_{\epsilon\in [k]} \mathcal{C}^{|\varepsilon|} f_\epsilon \,d\mu_{\ell,\gn}^{[k]} = 0$$ for all $\ell\ge 0$. Taking $\ell\rightarrow\infty$, the first property of the previous claim gives the desired result.
\end{proof}
\section{Limit formulae along IPs in $\mathbb{F}_p^\omega$-systems}\label{formula:sec}
\subsection{The structure theorem for $\mathbb{F}_p^\omega$-systems}
\begin{definition}[Polynomials and Weyl systems]
    Fix a prime $p$, let $V$ be a countable vector space over $\mathbb{F}_p$, let $\XX=(X,\mathcal{B},\mu,T)$ be a $V$-system, and let $K$ be a compact abelian group.
    \begin{itemize}
        \item[(1)] $\XX$ is called a \emph{rotational $V$-system}\footnote{Also known as Kronecker system in the literature.} if it is isomorphic to a compact abelian $p$-torsion group $U$, equipped with the Borel $\sigma$-algebra, the Haar measure, and an action of $V$ by rotations $T^\gamma u = \phi(\gamma)+u$, where $\phi:V\rightarrow U$ is a homomorphism.
        \item[(2)] The \emph{abelian extension} of $\XX$ by a cocycle $\rho:V\times X\rightarrow K$ is the system $\XX\times_\rho K$, defined as the product $X\times K$, with the product $\sigma$-algebra and the product measure (where $K$ is equipped with the Haar measure), and the action of $V$ is given by 
        $$T_\rho^\gamma (x,k) = (T^\gamma x, \rho(\gamma,x)+k).$$
        \item[(3)] A function $f:X\rightarrow K$ is called a polynomial of degree $<k$ for some $k\in \mathbb{N}\cup\{0\}$, if for every $\gamma_1,\dots,\gamma_k\in V$ we have $\partial_{\gamma_1}\dots\partial_{\gamma_k} f = 0_K$, where here $\partial_\gamma f(x) = f(T^\gamma x)-f(x)$.\footnote{When $K=\mathbb{C}$ we will assume that the derivative is multiplicative, i.e. $\partial_\gamma f(x) = f(T^\gamma x)\cdot\overline{f(x)}$.} Here we use the convention that $f=0$ is a polynomial of degree $<0$. More generally, a cocycle $\rho:V\times X\rightarrow K$ is called a polynomial of degree $<k$ if $\rho(\gamma,\cdot)$ is a polynomial of degree $<k$ for all $\gamma\in V$.
        \item[(4)] $\XX$ is called a \emph{Weyl $V$-system of order $k$} if there exists a rotational system $U_1$, and compact abelian $p$-torsion groups $U_2,\dots,U_{k}$ such that
        \begin{equation}\label{form}\XX = U_1\times_{\rho_1} U_2\times\dots\times_{\rho_{k-1}} U_k\end{equation} where $\rho_i:V\times (U_1\times\dots\times U_{i-1})\rightarrow U_i$ is a polynomial of degree $<i$. A Weyl system is called \emph{continuous} if in addition all the cocycles are continuous.
    \end{itemize}
\end{definition}
Let $V$ be a vector space over $\mathbb{F}_p$, and let $\XX=(X,\mathcal{B},\mu,T)$ be a $V$-system. Let $Z^0(\XX) = (Z^0(X),\mathcal{B}_0,\mu_0,T_0) $ denote the invariant factor of $\XX$. The ergodic decomposition gives rise to ergodic probability measures $\mu_s$ on $\XX$ such that
\begin{equation}\label{ergdec}
    \mu = \int_{Z^0(X)} \mu_s\,d\mu_0(s).
\end{equation}
Furthermore, since the action of $T$ on $Z^0(\XX)$ is trivial, we have $T^\gamma \mu_s = \mu_{T^\gamma s} = \mu_s$. Hence, we can view the ergodic component $\XX_s=(X,\mathcal{B},\mu_s,T)$ as an ergodic $V$-system. It is a classical result that for every $k\ge 0$ almost every ergodic component of $Z^k(\XX)$, is a system of order $k$ (cf. \cite{host2005nonconventional}). In other words, $Z^k(Z^k(\XX)_s) = Z^k(\XX)_s$.
We then have the following structure theorem of Bergelson, Tao and Ziegler \cite{bergelson2010inverse,btz2}.
\begin{theorem}[Structure theorem for $\mathbb{F}_p^\omega$-systems]\label{btz:structure}
    Let $V$ be a vector space over $\mathbb{F}_p$, and let $\XX=(X,\mathcal{B},\mu,T)$ be a $V$-system. Then almost every ergodic component of $Z^k(\XX)$ is isomorphic to a continuous Weyl system of order $k$.
\end{theorem}
Let $\XX$ be a Weyl system of order $k$ and write it in the form \eqref{form}. For the sake of notational simplicity, we abbreviate $\mathrm{HP}_{0,1,\dots,k}(\XX)$ for $\mathrm{HP}_{0,1,\dots,k}(U_1,\dots,U_k)$. We need two simple lemmas. The first is a vector space analogue of the unique ergodicity result of Parry \cite{Parryunique} for nilmanifolds and $\mathbb{Z}$-action. In our special case we can take advantage of the continuity of polynomials, which simplifies the argument significantly.
\begin{lemma}[Unique ergodicity]\label{uni-erg}
    Let $k\ge 1$, and let $V$ be a vector space over $\mathbb{F}_p$. Then any ergodic continuous Weyl $V$-system $\XX$ is uniquely ergodic. Namely, the tensor product of the Haar measures on the structure groups of $\XX$ is the unique $V$-invariant Borel probability measure on $\XX$.
\end{lemma}
\begin{proof}
    We induct on $k$. When $k=1$, $\XX=U_1$ is a rotational system. Write $T^\gamma x = x+\phi(\gamma)$, it follows from the Pontryagin duality that the image of $\phi$ is dense in $\XX$. Indeed, if not then $\phi(V)$ is annihilated by some non-trivial character $\chi:U_1\rightarrow S^1$, but then $\chi$ is a non-constant invariant function by contradiction. Thus, every $V$-invariant measure must be $U_1$-invariant and is therefore the Haar measure. Now, suppose that the claim holds for continuous Weyl systems of order $k-1$ and write $\XX = \XX_{k-1}\times_{\rho} U$ for some compact abelian group $U$, a continuous cocycle $\rho:V\times \XX_{k-1}\rightarrow U$ and a continuous Weyl $V$-system $\XX_{k-1}$ of order $k-1$. It suffices to show that the ergodic average of any continuous function in $C(\XX)$ converges pointwise (everywhere). By Stone-Weierstrass theorem, the algebra generated by $f_\chi(x,u) = f(x)\chi(u)$ where $f:\XX_{k-1}\rightarrow \mathbb{C}$ is continuous and $\chi\in \widehat{U}$, is dense in $C(\XX)$. Hence, it suffices to prove pointwise convergence for these functions. If $\chi$ is trivial, then the convergence follows by the induction hypothesis. Otherwise, suppose that $\chi\neq 1$, fix some $(x_0,u_0)\in \XX$, and let $\Phi=(\Phi_N)_{N\in\mathbb{N}}$ be any F\o lner sequence of $V$. Then, by the van der Corput lemma and the cocycle identity (i.e., $\rho(\gamma+\gamma',x) = \rho(\gamma,x)\cdot\rho(\gamma',T^\gamma x)$), we have
    \begin{align*}\limsup_{N\rightarrow\infty} \left|\E_{\gamma\in \Phi_N} T^\gamma f_\chi(x_0,u_0)\right|^2 &\le \limsup_{M\rightarrow\infty}\E_{\gamma'\in \Phi_M}\left|\limsup_{N\rightarrow\infty}\E_{\gamma\in \Phi_N}f_\chi(T^{\gamma+\gamma'}(x_0,u_0))\cdot \overline{f_\chi(T^\gamma (x_0,u_0))}\right|^2 \\&=\limsup_{M\rightarrow\infty}\E_{\gamma'\in \Phi_M}\left|\limsup_{N\rightarrow\infty}\E_{\gamma\in \Phi_N} (\partial_{\gamma'} f\cdot \chi\circ\rho_{\gamma'}) (T^\gamma x_0)\right|^2
    \end{align*}
    where $\partial_{\gamma'}f = f\circ T^{\gamma'}\cdot \overline{f}$. Since $ \partial_{\gamma'} f\cdot \chi\circ\rho_{\gamma'}$ is continuous, we see by the induction hypothesis that the inner $\limsup$ converges to the integral and we get
    \begin{align*}
        \limsup_{M\rightarrow\infty}\E_{\gamma'\in \Phi_M}\left|\int \partial_{\gamma'}f\cdot \chi\circ \rho_{\gamma'}\,d\mu\right|^2 &=\limsup_{M\rightarrow\infty}\E_{\gamma'\in \Phi_M}\left|\int \partial_{\gamma'}f_{\chi}\,d\mu\right|^2\\&=\norm{f_{\chi}}_{U^2(\XX)} =0.
    \end{align*}
    Therefore, $\lim_{N\rightarrow\infty}\E_{\gamma\in \Phi_N} T^\gamma f_\chi(x_0,u_0) = \int f_{\chi}\,d\mu = 0$, and this completes the proof.
\end{proof}
The following lemma is folklore. 
\begin{lemma}[The Junta factor is a rotational system]
      Let $V$ be a vector space over $\mathbb{F}_p$, let $k\ge 1$, and let $\XX$ be a $V$-system. Let $\gn$ be a sequence of elements in $V$, and suppose that $\XX$ is ergodic with respect to the action of $V_{\gn}:=\mathrm{Span}_{\mathbb{F}_p}\left(\gn\right)$. Then, $\mathrm{J}_{\gn}(\XX)$ is isomorphic to a rotational $V$-system.
\end{lemma}
\begin{proof}
    It is a classical result that the maximal rotational factor (the Kronecker factor) is the factor generated by all the eigenfunctions of $\XX$. Since $\mathrm{J}_{\gn}(\XX)$ is generated by a subset of these eigenfunctions, it is a sub-factor of the rotational factor of $\XX$. Finally, any factor of a rotational system is itself a rotational system. This completes the proof.
\end{proof}
Now suppose that $\XX$ is an ergodic Weyl system written as in \eqref{form}. From the previous lemma, we see that $\mathrm{J}_{\gn}(\XX)$ is a quotient of $U_1$. However, since $U_1$ is $p$-torsion, any closed subgroup has a complement. In other words, there exists a decomposition $U_1 = \mathrm{J}_{\gn}(\XX) \times U_1'$, for some closed subgroup $U_1'\le U_1$. This observation leads to the following definition.
\begin{definition}[The IP-Hall-Petresco groups]\label{HP:def}
      Let $V$ be a vector space over $\mathbb{F}_p$, let $k \ge 1$, and let $\gn\subseteq V$. Suppose that $\XX$ is a Weyl $V_{\gn}$-system of order $m$. For $\XX$ written in the form  \eqref{form} and $U_1 = \mathrm{J}_{\gn}(\XX)\times U_1'$, we define $$\mathrm{HP}_{0,1,\dots,k}^{\IP(\gn)}(\XX) = \mathrm{HP}_{0,0,\dots,0}(\mathrm{J}_{\gn}(\XX))\times \mathrm{HP}_{0,1,\dots,k}(U_1',U_2,\dots,U_m).$$
      As usual, this compact abelian group is equipped with the Borel $\sigma$-algebra, and Haar measure.
\end{definition}

Our main goal is to establish the following limit formula.
\begin{theorem}[Double limit formula for continuous Weyl systems]\label{Weyllimit}
    Let $V$ be a vector space over $\mathbb{F}_p$, let $k,m\geq 1$, and let $\XX$ be a continuous Weyl system of order $m$. Let $\gn$ be a sequence of elements in $V$, and let $\Phi=(\Phi_N)_{N\in\mathbb{N}}$ be an increasing F\o lner sequence. Finally, let $\mu = \int_{Z^0_{\gn}(\XX)} \mu_s\,d\mu^0_{\gn}(s)$ denote the ergodic decomposition of $\XX$ with respect to the action of $V_{\gn}\coloneqq\mathrm{Span}_{\mathbb{F}_p}(\gn)$.  
    Then for all bounded functions $f_0,\dots,f_k\in L^\infty(\mu)$, we have
    \begin{align*}\lim_{D\rightarrow\infty} \lim_{N\rightarrow\infty} &\E_{\IP_{\Phi_N^D}(\gn)} \int_{\XX} \prod_{i=0}^k T^{i\gamma} f_i\,d\mu \\&= \int_{Z^0_{\gn}(\XX)}\int_{\mathrm{HP}_{0,1,\dots,k}^{\IP(\gn)}(\XX_s)} \bigotimes_{i=0}^k f_i(\bm{x})\,d\mu_{\mathrm{HP}_{0,1,\dots,k}^{\IP(\gn)}(\XX_s)}(\bm{x})\,d\mu_{\gn}^0(s),
    \end{align*}
    where the ergodic components $\XX_s$, for almost all $s\in Z^0_{\gn}(\XX)$, are assumed to be written as in \eqref{form}.
\end{theorem}
\begin{proof}
    By the ergodic decomposition, it suffices to assume without loss of generality that the action of $V_{\gn}$ on $\XX$ is ergodic. Then, we need to show that
    \begin{equation}\label{formula}
    \begin{split}    
    \lim_{D\rightarrow\infty} \lim_{N\rightarrow\infty} &\E_{\gamma\in\IP_{\Phi_N^D}(\gn)} \int_{\XX} \prod_{i=0}^k T^{i\gamma} f_i\,d\mu \\&=\int_{\mathrm{HP}_{0,1,\dots,k}^{\IP(\gn)}(\XX)} \bigotimes_{i=0}^k f_i(\bm{x})\,d\mu_{\mathrm{HP}_{0,1,\dots,k}^{\IP(\gn)}(\XX)}(\bm{x}).
    \end{split}
\end{equation}
By Fourier analysis, it suffices to establish this claim in the case where $f_i$ are characters of the group $U_1\times\dots\times U_m$. Recall that $U_1 = \mathrm{J}_{\gn}(\XX)\times U_1'$, and so we can write $f_i = \chi_i \cdot f_i'$ where $\chi_i:\mathrm{J}_{\gn}(\XX)\rightarrow S^1$ is a character and $f_i'$ is orthogonal to $\mathrm{J}_{\gn}(\XX)$. Observe that
$$\prod_{i=0}^k T^{i \gamma} f_i = \prod_{i=0}^k \chi_i\cdot \prod_{i=0}^k T^{i\gamma}f_i'$$ for all $\gamma\in \IP_{\Phi_N^D}\big(\gn\big)$, whenever $D$ is sufficiently large. Hence, we may assume without loss of generality that all the $f_i$ are orthogonal to $\mathrm{J}_{\gn}(\XX)$ and we need to show that 
\begin{align*}\lim_{D\rightarrow\infty} \lim_{N\rightarrow\infty} &\E_{\gamma\in\IP_{\Phi_N^D}(\gn)} \int_{\XX} \prod_{i=0}^k T^{i\gamma} f_i\,d\mu \\&=\int_{\mathrm{HP}_{0,1,\dots,k}(U_1',U_2,\dots,U_{m})} \bigotimes_{i=0}^k f_i(\bm{x})\,d\mu_{\mathrm{HP}_{0,1,\dots,k}(U_1',U_2,\dots,U_{m})}(\bm{x}).
\end{align*}
Now, for the sake of notational simplicity, we set $\tilde{X}=\mathrm{HP}_{0,1,\dots,k}(\XX)$, and $\tilde{\mu} = \mu_{\mathrm{HP}_{0,1,\dots,k}(\XX)}$. For every $x\in X$, let $\tilde{X}_x \subseteq \tilde{X}$ denote the subset of all elements whose first coordinate is $x$. We let $\tilde{\mu}_x$ denote the Haar measure on $\tilde{X}_x$, and observe that $\tilde{\mu}=\int_X \tilde{\mu}_x\,d\mu(x)$. Furthermore, we observe that $V$ acts on $\tilde{X}_x$ by the transformation $\tau=I\times T\times\dots\times T^{k}$.\\
\noindent \textbf{Claim 1:} The action of $\tau$ on $\tilde{X}_x$ is ergodic for all $x\in X$.
\begin{proof}
This follows from the work of Bergelson, Tao and Ziegler \cite[Lemma 1.8]{btz2}. Indeed, they proved that for every F\o lner sequence $\Phi=(\Phi_N)_{N\in\mathbb{N}}$, we have
$$\lim_{N\rightarrow\infty}\E_{\gamma\in \Phi_N} \int_{X} \bigotimes_{i=0}^kg_i(\tau^\gamma(x,x,\dots,x))\,d\mu(x) = \int_{\tilde{X}} \bigotimes_{i=0}^kg_i\,d\tilde{\mu},$$ for all bounded functions $g_0,\dots,g_k\in L^\infty(\mu)$. Since $g_0$ is arbitrary, we see that for every $x\in X$, 
$$\E_{\gamma\in \Phi_N}  \bigotimes_{i=0}^kg_i(\tau^\gamma(x,x,\dots,x))\rightarrow \int_{\tilde{X}_x} \bigotimes_{i=0}^k g_i\,d\tilde{\mu}_x,$$ weakly as $N\rightarrow\infty$. By the mean ergodic theorem, this weak convergence is in fact a norm convergence. Furthermore, since $\tilde{X}_x$ is also a Weyl system, it is uniquely ergodic (Lemma \ref{uni-erg}). Thus, we have pointwise convergence of the ergodic average for all continuous $g_0,\dots,g_k$. Since the convergence of ergodic averages on continuous functions is equivalent to ergodicity, we deduce that the action of $\tau$ on $\tilde{X}_x$ is ergodic for all $x\in X$.
\end{proof}
\noindent \textbf{Claim 2:}  Let $x\in X$ be arbitrary. A function in $L^2(\tilde{\mu}_x)$ is measurable with respect to $\mathrm{J}_{\gn}(\tilde{X}_x)$ (where the latter is defined with respect to the action of $\tau$) if and only if it is measurable with respect to the meet $\mathrm{J}_{\gn}(\XX)^{k+1}\land \tilde{X}_x$.
\begin{proof}
The following direction is simple. If $f$ is measurable with respect to $\mathrm{J}_{\gn}(\XX)^k\land \tilde{X}_x$, then it can be approximated by functions of the form $f_0\otimes\dots\otimes f_k$ where $f_i$ is a $T$-eigenfunction whose eigenvalue $\chi_i\in \mathrm{J}_{\gn}$. But then, $f_0\otimes\dots\otimes f_k$ is a $\tau$-eigenfunction with eigenvalue $\prod_{i=0}^k \chi_i^i$. Since $\mathrm{J}_{\gn}$ is a group, we deduce that $f$ is measurable with respect to $\mathrm{J}_{\gn}(\tilde{X}_x)$, as required. To prove the other inclusion, we make an observation. For every $\ell>0$, let $\mathrm{J}_{\ell,\gn}(\tilde{X}_x)$ denote the factor generated by the eigenfunctions whose eigenvalue is trivial on $\gamma_1,\dots,\gamma_{\ell}$ (i.e., $\xi(\gamma_i)=1$ for all $i=1,\dots,\ell$). Since $\mathrm{J}_{\gn}(\tilde{X}_x) = \bigvee_{\ell\ge 0} \mathrm{J}_{\ell,\gn}(\tilde{X}_x)$, it suffices to prove the claim where $\mathrm{J}_{\gn}$ is replaced with $\mathrm{J}_{\ell,\gn}$ for some fixed $\ell>0$.

Now let $V_{>\ell} =\mathrm{Span}_{\mathbb{F}_p}(\{\gamma_{n+\ell}:n\in\mathbb{N}\})$, and let $V_{\ell}$ denote its complement such that $V_{\gn} = V_{\ell} \oplus V_{>\ell}$. Then any function $f$, that is measurable with respect to $\mathrm{J}_{\ell,\gn}(\tilde{X}_x)$, is $V_{>\ell}$-invariant. 

Let $h\in L^2(\mu)$, and let $U_1$ denote the rotational factor of $\XX$. Then as before, we can write $U_1 = U_1^{(\ell)}\times U_1^{(>\ell)}$, where $U_1^{(\ell)} = \mathrm{J}_{\ell,\gn}(\XX)$. We see that $h$ is measurable with respect to $\mathrm{J}_{\ell,\gn}(\XX)$, if and only if it is $V_{>\ell}$-invariant and thus also invariant to $U_1^{(>\ell)}\times U_2\times\dots\times U_m$. We deduce that the ergodic components of $\XX$ with respect to $V_{>\ell}$ are of the form $\{x\}\times \XX/U_1^{(\ell)}$, for all $x\in U_1^{(\ell)}$. Hence, since the induced action of $V_{>\ell}$ on these components is ergodic, it follows from the work of Bergelson, Tao, and Ziegler that the induced action of $V_{>\ell}$ by $\tau$ on the components $(x_0,\dots,x_k)+\mathrm{HP}_{0,1,\dots,k}(U_1^{(>\ell)},U_2,\dots,U_m)$, where $x_i\in U_1^{(\ell)}$ for all $i=0,\dots,k$, is ergodic. We conclude that a function $f$, measurable with respect to $\mathrm{J}_{\ell,\gn}(\tilde{X}_x)$, must be invariant to the action of $\mathrm{HP}_{0,1,\dots,k}(U_1^{(>\ell)},U_2,\dots,U_m)$, and is therefore measurable with respect to $\mathrm{HP}_{0,1,\dots,k}(U_1^{(\ell)}) = \mathrm{J}_{\ell,{\gn}}(\XX)^{k+1}\cap \tilde{X}_x$.
\end{proof}
We return to the proof of the theorem. Recall that our functions $f_0,\dots,f_k$ are orthogonal to the Junta factor $\mathrm{J}_{\gn}(\XX)$, and so from Claim $2$ we have that
$$\E\left(\bigotimes_{i=0}^k f_i \mid \mathrm{J}_{\gn}(\tilde{X}_x)\right)  = \E\left(\bigotimes_{i=0}^k f_i \mid \mathcal{I}_{x,\tau}\right),$$ where $\mathcal{I}_{x,\tau}$ is the invariant $\sigma$-algebra. From Claim $1$, the right hand side of the equation above is equal to the integral with respect to $\tilde{\mu}_x$. Combining this with Corollary \ref{refined}, we see that 
\begin{align*}\lim_{D\rightarrow\infty} \lim_{N\rightarrow\infty} \E_{\gamma\in\IP_{\Phi_N^D}(\gn)} \int_{\XX}& \prod_{i=0}^k T^{i\gamma} f_i\,d\mu\\ &= \lim_{D\rightarrow\infty}\lim_{N\rightarrow\infty} \E_{\gamma\in \IP_{\Phi_N^D}(\gn)} \int_{\tilde{X}} \tau^\gamma \bigotimes_{i=0}^k f_i\,d\tilde{\mu}\\&= \int_X \int_{\tilde{X}_x} \lim_{D\rightarrow\infty} \lim_{N\rightarrow\infty}\E_{\gamma\in\IP_{\Phi_N^D}(\gn)} \tau^\gamma \bigotimes_{i=0}^k f_i \,d\tilde{\mu}_x\,d\mu(x)\\&= \int_X \E\left(\bigotimes_{i=0}^k f_i \mid \mathrm{J}_{\gn}(\tilde{X}_x)\right)\,d\mu(x)\\&= \int_X \int_{\tilde{X}_x} \bigotimes_{i=0}^k f_i\,d\tilde{\mu}_x\,d\mu(x),
\end{align*}
and the last summand is equal (by construction) to the right hand side of \eqref{formula}. This completes the proof.
\end{proof}
\section{Proof of Theorem \ref{qualitativemain}}\label{proof}
 Recall that $c_{\IP}=c_{\IP}^{\rec}\le c^{\rec}=c$. We split the proof into multiple inequalities. We begin by proving the following one.
\begin{lemma}\label{c<d}
Let $p$ be a prime number and $k\le p$ an integer. Then for every $\delta\in [0,1]$ we have $c^{\rec}(k,\delta,\mathbb{F}_p^\omega) \le \dAPk(\delta)$. Thus, all the constants $c_{\IP},c_{\IP}^{\rec}$ and $c$ are bounded above by $\dAPk$.
\end{lemma}
\begin{proof}
Fix some $\delta\in[0,1]$ and let $\varepsilon>0$ be arbitrary. It suffices to show that
    $$c^{\rec}(k,\delta,\mathbb{F}_p^\omega) < \dAPk(\delta)+\varepsilon.$$
    By Lemma \ref{function-set}, we can find compact abelian $p$-torsion groups $U_1,\dots,U_{m}$ and a subset $A\subseteq U_1\times\dots\times U_m$ such that \begin{equation}\label{c<df}
    S^k_{U_1,\dots,U_m}(1_A) < \dAPk(\delta)+\varepsilon/2.
    \end{equation}   
   Let $G = \mathrm{HP}_{0,1,\dots,k-1}(U_1,\dots,U_m)$ and consider the system $X = G^{\mathbb{N}}\times G$ equipped with the product measure and the action of $V=\mathbb{F}_p^\omega$ induced by the maps $T_{e_i}(\bm{g},h) = (\bm{g},h+g_i)$ where $\bm{g}=(g_1,g_2,\dots)\in G^{\mathbb{N}}$ and $h\in G$. Observe that $A$ is embedded in $G$ as trivial progressions. 
   Now, consider the set $B = G^{\mathbb{N}} \times A$, then for every $0\neq\gamma\in V$, we have
    \begin{align*}
        \mu\left(\bigcap_{i=0}^{k-1} T^{-i\gamma} B\right) = \int_{G^{\mathbb{N}}}\int_G \prod_{i=0}^{k-1} 1_A(h+i\cdot\gamma\cdot \bm{g})\,d\mu_G(h)\,d\mu_{G^\mathbb{N}}(\bm{g}).
    \end{align*}
    Whenever $\gamma\neq 0$, we can, for $\mu_{G^{\mathbb{N}}}$-almost every $\bm{g}\in G^{\mathbb{N}}$ perform a change of variables giving
    $$ \mu\left(\bigcap_{i=0}^{k-1} T^{-i\gamma} B\right) = S^k_{U_1,\dots,U_m}(1_A) <\dAPk(\delta)+\varepsilon/2.$$
    This completes the proof.
\end{proof}
\begin{proposition}\label{convexconstants}
    The maps $\delta\mapsto c_{\IP}^{\rec}(k,\delta,\mathbb{F}_p^\omega),\delta\mapsto c_{\IP}(k,\delta,\mathbb{F}_p^\omega),\delta\mapsto c^{\rec}(k,\delta,\mathbb{F}_p^\omega)$ and $\delta\mapsto c(k,\delta,\mathbb{F}_p^\omega)$ are convex for any fixed choice of $k$ and $p$. Thus, all of these constants are bounded from above by $\cdAPk(\delta)$.
\end{proposition}
\begin{proof}
    Let $\lambda\in[0,1]$ and let $\delta = \lambda\cdot\delta_1+(1-\lambda)\cdot \delta_2$ where $\delta,\delta_1,\delta_2\in[0,1]$. Consider any two $V$-systems $\XX_1=(X_1,\mu_1,T_1),\XX_2=(X_2,\mu_2,T_2)$ and any two subsets $A_1\subseteq X_1$ and $A_2\subseteq X_2$ with measures $\mu_1(A_1)\ge \delta_1$ and $\mu_2(A_2)\ge \delta_2$. We can form a system $\XX_\lambda = \XX_1\sqcup \XX_2$ equipped with the action $$T^\gamma x = \begin{cases} T^\gamma_1 x & x\in X_1
    \\
    T_2^\gamma x & x\in X_2\end{cases}$$ and the measure $\mu = \lambda\mu_1 + (1-\lambda)\mu_2$. Then we can consider the set $A=A_1\sqcup A_2$ and observe that 
    $$\mu\left(\bigcap_{i=0}^{k-1} T^{-i\gamma} A \right)=\lambda\cdot \mu_1\left(\bigcap_{i=0}^{k-1} T_1^{-i\gamma} A_1 \right)+(1-\lambda)\cdot \mu_2\left(\bigcap_{i=0}^{k-1} T_2^{-i\gamma} A_2 \right),$$ meanwhile $\mu(A) \geq \lambda\delta_1+(1-\lambda)\delta_2 = \delta$. We conclude that  $\delta\mapsto c_{\IP}^{\rec}(\delta)$ and $\delta\mapsto c^{\rec}(\delta)$ are convex. The other two constants are convex since they are equal to these constants by the inverse Furstenberg correspondence principle.
\end{proof} 
Finally, we prove the lower bound. Since $c_{\IP}^{\rec}$ is the minimal constant, it suffices to prove the following inequality.
\begin{proposition}\label{c>d}
Let $p$ be a prime and $k\le p$ an integer. Then for every $\delta\in[0,1]$ we have
$$c_{\IP}^{\rec}(k,\delta,\mathbb{F}_p^\omega) \ge \cdAPk(\delta).$$   
\end{proposition}
\begin{proof}

    Assume by contradiction that this claim fails for some $\delta>0$ and some $k\le p$. Thus, there exists some $\varepsilon>0$ sufficiently small, such that $c_{\IP}^{\rec}(k,\delta,\mathbb{F}_p^\omega) < \cdAPk(\delta)-\varepsilon$.
    From the definition, there exists an $\mathbb{F}_p^\omega$-system $\XX=(X,\mathcal{B},\mu,T)$, a measurable subset $A\subseteq X$ with $\mu(A)\ge \delta$, and some $\gn\subseteq \mathbb{F}_p^\omega$, such that 
    \begin{equation}\label{small}
        \mu(A\cap T^{-\gamma}A\cap\dots\cap T^{-(k-1)\gamma}A) < \cdAPk(\delta)-\varepsilon
    \end{equation}
    for all $\gamma\in \IP(\gn)\setminus \{0\}$. But then we must have
   $$
    \lim_{N\rightarrow\infty} \EIP \mu(A\cap T^{-\gamma}A\cap\dots\cap T^{-(k-1)\gamma}A) \le \cdAPk(\delta)-\varepsilon
    $$
    for all increasing F\o lner sequences $\Phi=(\Phi_N)_{N\in\mathbb{N}}$. In particular,
  \begin{equation}\label{contradictionfinal}
    \lim_{D\rightarrow\infty}\lim_{N\rightarrow\infty} \E_{\IP_{\Phi_N^D}(\gn)} \mu(A\cap T^{-\gamma}A\cap\dots\cap T^{-(k-1)\gamma}A) \le \cdAPk(\delta)-\varepsilon
    \end{equation}
    
    Let $V=V_{\gn}$ denote the subspace of $\mathbb{F}_p^\omega$ generated by $\gn$, and let $(Z^0_{\gn}(\XX),\mu^0_{\gn})$ be the $V_{\gn}$-invariant factor. The ergodic decomposition gives rise to ergodic measures $\mu_s$ on $\XX$ such that 
    \begin{equation}\label{ergodic-decomposition}
    \mu = \int_{Z^0_{\gn}(\XX)} \mu_s\,d\mu^0_{\gn}.
    \end{equation}
    Let $\XX_s$ denote the $s$-component. Then for every $s\in Z^0_{\gn}$, we can write $1_A = f_s + g_s$, where $f_s = \E_{\mu_s}(1_A \mid Z^{k-1}(\XX_s,V_{\gn}))$ and $g_s = 1_A-f_s$ is the orthogonal complement. Here $Z^{k-1}(\XX_s,V_{\gn})$ is the Host--Kra factor of the ergodic component. From Lemma \ref{measurecontrol} we see that $\|g_s\|_{U^{k}_{\IP}(\XX_s,\gn)} = 0,$ and so by \eqref{seminormcontrol}, we have 
    $$\lim_{N\rightarrow\infty} \EIP \mu_s(A\cap T^{-\gamma}A\cap\dots\cap T^{-(k-1)\gamma}A) = \lim_{N\rightarrow\infty} \EIP \int \prod_{i=0}^{k-1} T^{i\cdot \gamma} f_s \,d\mu_s$$ for $\mu^0_{\gn}$-almost every $s\in Z^0_{\gn}$.
    Combining with \eqref{ergodic-decomposition}, and \eqref{formula} we see that 
    \begin{equation}\label{reductiontoHKfactor}
    \begin{split}
           &\lim_{D\rightarrow\infty} \lim_{N\rightarrow\infty} \E_{\IP_{\Phi_N^D}(\gn)} \mu(A\cap T^{-\gamma}A\cap\dots\cap T^{-(k-1)\gamma}A)  \\&=  \int_{Z^0_{\gn}(\XX)}\left(\lim_{D\rightarrow\infty}\lim_{N\rightarrow\infty}\E_{\IP_{\Phi_N^D}(\gn)}  \int \prod_{i=0}^{k-1} T^{i\cdot \gamma} f_s \,d\mu_s\right)\,d\mu^0_{\gn}(s) \\ &=\int_{Z^0_{\gn}(\XX)}\left(\int_{\mathrm{HP}_{0,1,\dots,k-1}^{\IP(\gn)}(\XX_s)} \bigotimes_{i=0}^{k-1} f_s(\bm{x})\,d\mu_{\mathrm{HP}_{0,1,\dots,k-1}^{\IP(\gn)}(\XX_s)}(\bm{x})\right)\,d\mu_{\gn}^0(s).
        \end{split}
    \end{equation}
    Let $\mathcal{U}'_s\coloneqq U_{1,s}'\times\dots\times U_{k-1,s}$ and write $\XX_s = \mathrm{J}_{\gn}(\XX_s) \times \mathcal{U}'_s$.
    Recall that $\mathrm{HP}^{\IP(\gn)}_{0,1,\dots,k-1}(\XX_s) = \mathrm{HP}_{0,0,\dots,0}(\mathrm{J}_{\gn}(\XX_s))\times \mathrm{HP}_{0,1,\dots,k-1}(U'_{1,s},\dots,U_{k-1,s})$ (apply Definition \ref{HP:def} to $\XX_s$). For $x\in X_s$, write $x=(x_{J},x')$ where $x_{J}\in \mathrm{J}_{\gn}(\XX_s)$ and $x'\in \mathcal{U}'_s$, and write $f_{s,x_J}(x') = f_s(x_J,x')$. Then, \eqref{reductiontoHKfactor} equals
   \begin{align*}
          &\int_{Z^0_{\gn}(\XX)} \int_{\mathrm{J}_{\gn}(\XX_s)}S_{U_{1,s}',\dots,U_{k-1,s}}^k(f_{s,x_J})\,d\mu_{\mathrm{J}_{\gn}(\XX_s)}(x_J)\,d\mu^0_{\gn}(s) \\ &\ge \int_{Z^0_{\gn}(\XX)} \int_{\mathrm{J}_{\gn}(\XX_s)} \dAPk \left(\int_{\mathcal{U}'_s} f_{s,x_J}\,d\mu_{\mathcal{U}'_s}\right)\,d\mu_{\mathrm{J}_{\gn}(\XX_s)}(x_J)\,d\mu^0_{\gn}(s)\\&\ge \int_{Z^0_{\gn}(\XX)} \int_{\mathrm{J}_{\gn}(\XX_s)}\cdAPk\left(\int_{\mathcal{U}'_s} f_{s,x_J}\,d\mu_{\mathcal{U}'_s}\right)\,d\mu_{\mathrm{J}_{\gn}(\XX_s)}(x_J)\,d\mu^0_{\gn}(s)\\&\ge\cdAPk\left(\int_{Z^0_{\gn}(\XX)} \int_{\mathrm{J}_{\gn}(\XX_s)}\int_{\mathcal{U}'_s} f_{s,x_J}\,d\mu_{\mathcal{U}'_s}\,d\mu_{\mathrm{J}_{\gn}(\XX_s)}(x_J)\,d\mu^0_{\gn}(s)\right)\\&=
          \cdAPk\left(\delta\right).
    \end{align*}
    This clearly contradicts \eqref{contradictionfinal}, as required.
\end{proof}

 \appendix
 \section{Random sets}\label{randomsec}
Throughout, fix a prime number $p$ and let $k\le p$. Let $V$ be a countable-dimensional vector space over $\mathbb{F}_p$, and let $\Phi=(\Phi_N)_{N\in\mathbb{N}}$ be a F\o lner sequence for $V$. A \emph{distribution function} is a map $F:V\rightarrow [0,1]$ such that the limit $\lim_{N\rightarrow\infty}\E_{\gamma\in \Phi_N} F(\gamma) =\delta_F$ exists. Every such function gives rise to a random set; indeed, each value of $F(\gamma)$ defines a measure $\mu_{\gamma}$ on $\{0,1\}$ which assigns the value $F(\gamma)$ to $\{1\}$ and $1-F(\gamma)$ to $\{0\}$. Let $\Omega_F = \{0,1\}^V$ be the product space equipped with the product measure $\mu_F = \otimes_{\gamma\in V} \mu_\gamma$.  The coordinate maps $X_\gamma : \Omega_F\rightarrow \{0,1\}$ form an independent sequence of random variables where $X_\gamma$ is Bernoulli with probability $F(\gamma)$, for all $\gamma\in V$. This gives rise to a \emph{random set}, that is, a random variable $E:\Omega_F\rightarrow \mathcal{P}(V)$ defined by 
 \begin{equation}\label{randomE}
     E = \{\gamma\in V : X_\gamma =1\}.
 \end{equation}
\begin{proposition}\label{randomset}
Let $F$ be a distribution function, and let $E$ be as above. Suppose that $\Phi=(\Phi_N)_{N\in\mathbb{N}}$ is tempered, and that $|\Phi_N|>N$ for all $N$. If
$$\lim_{N\rightarrow\infty} \E_{\gamma\in \Phi_N} F(\gamma)F(\gamma+a)\dots F(\gamma+(k-1)a)$$ exists for all $a$ in some subset $A\subseteq\mathbb{F}_p^\omega$, then $\mu_F$-almost surely we have
    \begin{itemize}
        \item[(i)] $d_{\Phi}(E)=\delta_F$.
        \item[(ii)] For every $0\neq a\in A$, $d_{\Phi}(E\cap E-a \cap \dots\cap E-(k-1)a) = \lim_{N\rightarrow\infty} \E_{\gamma\in \Phi_N} F(\gamma)F(\gamma+a)\dots F(\gamma+(k-1)a)$.
    \end{itemize}
 \end{proposition}
 \begin{proof}
Since the countable intersection of co-null sets is co-null, it suffices to prove that each one of these conditions is satisfied almost surely, separately. We start with $(i)$. We have
$$d_{\Phi}(E) = \lim_{N\rightarrow\infty} \frac{|E\cap \Phi_N|}{|\Phi_N|} = \lim_{N\rightarrow\infty} \E_{\gamma\in \Phi_N} X_\gamma = \lim_{N\rightarrow\infty} \E_{\gamma\in \Phi_N} \E(X_\gamma)$$ where the last equality holds almost surely by the strong law of large numbers (where the version along tempered F\o lner sequences follows from the pointwise ergodic theorem). Thus, $d_{\Phi}(E)=\delta_F$ almost surely, by definition. To prove $(ii)$, it suffices to fix some $0\neq a\in A$ and show that
$$d_{\Phi}(E\cap E-a \cap \dots\cap E-(k-1)a) =  \lim_{N\rightarrow\infty}  \E_{\gamma\in \Phi_N} F(\gamma)F(\gamma+a)\dots F(\gamma+(k-1)a)$$ almost surely for that particular $a$. For any $\gamma\in V$, let $$Y_{\gamma} = \prod_{i=0}^{k-1} X_{\gamma+ia}.$$ Since $a\neq 0$, and $k\le p$, we see that $(X_{\gamma+ia})_{i=0}^{k-1}$ are independent random variables and so \begin{equation}\label{EY}
    \E Y_\gamma = \prod_{i=0}^{k-1}\E X_{\gamma+ia} = \prod_{i=0}^{k-1} F(\gamma+ia).
\end{equation}
Throughout, we let $\mu_{a,\gamma}$ denote the element on the right-hand side of the equation above.
Observe that we have
$$d_{\Phi}(E\cap E-a\cap\dots\cap E-(k-1)a) = \lim_{N\rightarrow\infty} \E_{\gamma\in \Phi_N} Y_\gamma.$$
Our next goal is to mimic the proof of the (strong) law of large numbers in this particular scenario (note that $(Y_\gamma)_{\gamma\in V}$ may not be independent). Let $$\overline{Y}_N = \E_{\gamma\in \Phi_N} (Y_\gamma-\mu_{a,\gamma}).$$ 
We have $\E(\overline{Y}_N) =  \E_{\gamma\in \Phi_N} \left(\E(Y_\gamma)-\mu_{a,\gamma} \right)= \E_{\gamma\in \Phi_N} \left(\prod_{i=0}^{k-1} F(\gamma+ia)-\mu_{a,\gamma}\right) = 0$. 
Our goal is to show that
$$\mathbb{P}(\lim_{N\rightarrow\infty} \overline{Y}_N = 0) =1.$$ Fix some $\varepsilon>0$, and let
$A_N = \{|\overline{Y}_N|\ge \varepsilon\}$. If $$\sum_{N=1}^\infty \mathbb{P}(A_N)<\infty, $$ then the claim follows by Borel--Cantelli. To estimate $A_N$, we need a claim.\\
\textbf{Claim:} For every even number $t$, we have $\E((\overline{Y}_N)^t) \le \frac{C(k,t)}{|\Phi_N^{t/2}|}$ for some absolute constant depending only on $k,t$.\\
\textbf{Proof:} We compute
$$\E\left((\overline{Y}_N)^t\right) = \E_{(\gamma_1,\dots,\gamma_t)\in \Phi_N^t}\E\left(\prod_{i=1}^t (Y_{\gamma_i}-\mu_{a,\gamma_i})\right). $$
Write $$B_t = \left\{(\gamma_1,\dots,\gamma_t)\in \Phi_N^t : \forall_{1\le i\le t} \exists_{j\neq i}\text{ such that } \gamma_i-\gamma_j \in \{da:d\in[-(k-1),k-1]\} \right\}$$
and observe that if $(\gamma_1,\dots,\gamma_t)\notin B_t$, then there exists some random variable $Y_{\gamma_i}$ that is independent of all the rest. In that case, we have $\E\left(\prod_{i=1}^t (Y_{\gamma_i}-\mu_{a,\gamma_i})\right) = \E\left(Y_{\gamma_i}-\mu_{a,\gamma_i}\right)\cdot\E\left(\prod_{j\neq i}(Y_{\gamma_j}-\mu_{a,\gamma_j})\right) = 0$. On the other hand, since every coordinate has another coordinate satisfying the condition in the definition of $B_t$, we must be able to split the elements $(\gamma_1,\dots,\gamma_t)$ into pairs, where the elements of each pair are close to each other. Therefore, $|B_t| < C(k,t)\cdot |\Phi_N^{t/2}|$, where $C(k,t)$ is the number of ways of splitting $(\gamma_1,\dots,\gamma_t)$ into pairs times $(2k+1)$ instances for $\gamma_i-\gamma_j$ for every pair. For all those $(\gamma_1,\dots,\gamma_t)\in B_t$, we have the obvious bound $\E\left(\prod_{i=1}^t (Y_{\gamma_i}-\mu_{a,\gamma_i})\right)\le 1$. We conclude that
\begin{align*}
     \E\left((\overline{Y}_N)^t\right)&=  \E_{(\gamma_1,\dots,\gamma_t)\in \Phi_N^t}\E\left(\prod_{i=1}^t (Y_{\gamma_i} - \mu_{a,\gamma_i})\right) \\&\le\frac{1}{|\Phi_N|^t}\sum_{(\gamma_1,\dots,\gamma_t)\in  B_t} 1\\
    &= \frac{|B_t|}{|\Phi_N|^t} \\&= \frac{C(k,t)}{|\Phi_N|^{t/2}}.
\end{align*}
This proves the claim. $\square$

Returning to the proof of the proposition, by Markov's inequality we have
$$\mathbb{P}(A_N)\le \frac{1}{\varepsilon^4}\cdot\E\left((\overline{Y}_N)^4\right)\le \frac{C(k,4)}{\varepsilon^4\cdot |\Phi_N|^2}.$$ Hence, since $|\Phi_N|>N$, this is a summable sequence and the claim follows by Borel--Cantelli.
 \end{proof}

 In a similar manner, we have the following result.
 \begin{lemma}\label{randombound}
     For every $k\ge 3$, and every $\delta\in[0,1]$, we have
     $$c(k,\delta,\mathbb{F}_p^\omega) \le \delta^k.$$
 \end{lemma}
 \begin{proof}
 It is tempting to set $F(x) = \delta$ for all $x\in \mathbb{F}_p^\omega$ and try to use the previous proposition. However, unfortunately here the F\o lner sequence $\Phi=(\Phi_N)_{N\in\mathbb{N}}$ is not assumed tempered. Instead, we rely on the inverse correspondence principle and show that $c^{\mathrm{rec}}(k,\delta,\mathbb{F}_p^\omega)\le \delta^k$.

Let $X = \{0,1\}^V$ equipped with the Borel $\sigma$-algebra, the measure $\mu = \mu_{\delta}^{\otimes V}$ where $\mu_{\delta}(x) = \begin{cases} \delta & \text{if } x=1 \\ 1-\delta & \text{if }x=0\end{cases}$, and the action of $V$ by shifts $\sigma^v x(u)=x(u+v)$ for all $v,u\in V$. Let $A = \{ x\in X : x(0_V) = 1\}$ be the cylinder set and observe that $\mu(A)=\delta$. Moreover, we have for all non-zero $v\in V$ that
$$\mu\left(\bigcap_{i=0}^{k-1} \sigma^{iv}A\right) = \delta^k.$$ Since 
$c^{\mathrm{rec}}(k,\delta,\mathbb{F}_p^\omega)\le \sup_{v\neq 0}\mu\left(\bigcap_{i=0}^{k-1} \sigma^{iv}A\right),$ the claim follows.
 \end{proof}
 For the sake of the proof of Lemma \ref{function-set} below, we will need the following generalization of Lemma \ref{randomset}. To avoid confusion, given a function $f$ we let $\bigotimes_{i=0}^kf(\bm{x})=\prod_{i=0}^k f(x_i)$ denote the $(k+1)$-fold tensor product of $f$ evaluated at $\bm{x}=(x_0,...,x_k)$.
 \begin{lemma}\label{randomV2}
     Let $m\in \mathbb{N}$ be an integer. Let $V$ be a countable-dimensional vector space over $\mathbb{F}_p$ and write $V=\mathbb{F}_p^\omega$ with respect to some basis. Let $F:V^m\rightarrow[0,1]$ be a distribution function, and let $E\subseteq V^m$ be a random set associated with this distribution. Let $\Phi=(\Phi_N)_{N\in\mathbb{N}}$ be a tempered F\o lner sequence and suppose that $|\Phi_N|>N$ for all $N\in\mathbb{N}$. If the limit
     $$\lim_{N\rightarrow\infty} \E_{\bm{x}\in \mathrm{HP}_{0,1,\dots,k}(\mathbb{F}_p^N,\dots,\mathbb{F}_p^N)}\bigotimes_{i=0}^k F(\bm{x})$$ exists, then $\mu_F$-almost surely we have
     \begin{itemize}
         \item[(i)] $\lim_{N\rightarrow\infty} \frac{|E\cap \prod_{i=1}^m \mathbb{F}_p^N|}{p^{N\cdot m}} =\delta_F$.
         \item[(ii)]
         $\lim_{N\rightarrow\infty} \E_{\bm{x}\in \mathrm{HP}_{0,1,\dots,k}(\mathbb{F}_p^N,\dots,\mathbb{F}_p^N)}\bigotimes_{i=0}^k1_E(\bm{x}) =\lim_{N\rightarrow\infty} \E_{\bm{x}\in \mathrm{HP}_{0,1,\dots,k}(\mathbb{F}_p^N,\dots,\mathbb{F}_p^N)}\bigotimes_{i=0}^k F(\bm{x})$.
     \end{itemize}
 \end{lemma}
 \begin{proof}
     The first property follows from the exact same proof as in Lemma \ref{randomset}. To prove $(ii)$ we use a similar argument as in the proof of Lemma \ref{randomset}. For the sake of simplicity of notation, we denote $\mathrm{HP}_{k}(W^m) = \mathrm{HP}_{0,1,\dots,k}(W,W,\dots,W)$ for any subspace $W\subseteq V$. For every $\bm{t}\in \mathrm{HP}_{k}(V^m)$, let $Y_{\bm{t}} = \prod_{i=1}^{k}X_{t_i}$, where $t_i$ is the $i^{\mathrm{th}}$-coordinate of $\bm{t}=(t_1,\dots,t_k)$.  Again, for the sake of simplicity of notation we write $\mathrm{HP}'_{k}((\mathbb{F}_p^N)^m))$ for the subset of $\mathrm{HP}_k((\mathbb{F}_p^N)^m)$ consisting of all elements whose coordinates are pairwise distinct. As $N\rightarrow\infty$ we have,
\begin{equation}\label{different}\frac{|\mathrm{HP}'_{k}((\mathbb{F}_p^N)^m)|}{|\mathrm{HP}_{k}((\mathbb{F}_p^N)^m)|}=1,
     \end{equation} so we can focus only on those coordinates. Now, if all the coordinates of $\bm{t}$ are different, then $\E(Y_{\bm{t}}) = \prod_{i=1}^k \E (X_{t_i}) = \prod_{i=1}^k F(t_i)$. We denote the right hand side by $\mu_{\bm{t}}$. Let $\overline{Y}_N = \E_{\bm{t}\in \mathrm{HP}'_{k}((\mathbb{F}_p^N)^m)} (Y_{\bm{t}}-\mu_{\bm{t}})$ and observe that by \eqref{different} we have $\E(\overline{Y}_N) =0. $
     Using the same argument as in Lemma \ref{randomset}, it suffices to bound the $4$-th moment. As in the previous proof, the number of quadruples $(\bm{t^{(1)}},\bm{t^{(2)}},\bm{t^{(3)}},\bm{t^{(4)}})$, where at least one $Y_{\bm{t}^{(i)}}$ is independent on the others is $C(k,4)\cdot|\Phi_N|^2$, for some constant $C(k,4)$, that is larger than the constant in the proof of Lemma \ref{randomset}, but still bounded for a fixed value of $k$. For these quadruples, we get $\E(\prod_{i=1}^4(Y_{\bm{t^{(i)}}}-\mu_{\bm{t}^i}))=0$, and for all other quadruples we can bound the expectation by $1$. Using the exact same argument as in the proof of Lemma \ref{randomset}, we deduce that $\E\left((\overline{Y}_N)^4\right)<\frac{C(k,4)}{|\Phi_N|^2}$, and the claim follows by Borel--Cantelli.
 \end{proof}
 \section{Properties of $\dAPk$}
 In this section we study properties of $\dAPk$ and the constants $c,c_{\IP},c^{\rec}, c_{\IP}^{\rec}$. Throughout, $k\ge 3$ is an integer. 
 \begin{proposition}\label{submult}
    The function $\dAPk:[0,1]\rightarrow [0,1]$ is sub-multiplicative. Namely, $\dAPk(x_1\cdot x_2)\le \dAPk(x_1)\cdot \dAPk(x_2)$, for all $x_1,x_2\in [0,1]$.
 \end{proposition}
\begin{proof}
Fix some $x_1,x_2\in[0,1]$ and let $\varepsilon>0$ be arbitrarily small. By definition we can find $A_1\subseteq \mathcal{U}_1$ and $A_2\subseteq \mathcal{U}_2$, such that
\begin{equation}
    S_{\mathcal{U}_i}(A_i) < \dAPk(x_i)+\varepsilon
\end{equation} for all $i=1,2$, where $\mathcal{U}_1,\mathcal{U}_2$ are products of finitely many $p$-torsion groups, and $\mu_1(A_1)\ge x_1, \mu_2(A_2)\ge x_2$, where $\mu_1$ and $\mu_2$ are the Haar measures on $\mathcal{U}_1$ and $\mathcal{U}_2$, respectively. Then $A_1\times A_2\subseteq \mathcal{U}_1\times \mathcal{U}_2$ is a subset of measure $(\mu_1\otimes\mu_2)(A_1\times A_2)\ge x_1\cdot x_2$ and $\dAPk(x_1\cdot x_2)\le S_{\mathcal{U}_1\times\mathcal{U}_2}(A_1\times A_2) = S_{\mathcal{U}_1}(A_1)\cdot S_{\mathcal{U}_2}(A_2) < \dAPk(x_1)\cdot \dAPk(x_2) + o_{x_1,x_2,\varepsilon\rightarrow0}(1)$. Since $\dAPk(x)\le 1$ for all $x$ and $\varepsilon>0$ is arbitrary, we see that $\dAPk(x_1\cdot x_2)\le \dAPk(x_1)\cdot \dAPk(x_2)$, as required.
\end{proof}
\begin{proposition}
      The function $\dAPk:[0,1]\rightarrow [0,1]$ is continuous.
\end{proposition}
\begin{proof}
Continuity at $x=1$ follows from the union bound. Indeed, if $f_i:X\rightarrow[0,1]$ satisfy $\int f_i\,d\mu>1-\varepsilon_i$, then since $\prod_{i=1}^k f_i \ge \sum_{i=1}^k f_i - (k-1)$, we have that $\int \prod_{i=1}^k f_i\,d\mu \ge 1-\sum_{i=1}^k \varepsilon_i$. In other words, $\dAPk(1-\varepsilon)\ge 1-k\varepsilon\rightarrow 1$ as $\varepsilon\rightarrow 0$.

Now let $x\in[0,1)$. We show that for every $\varepsilon>0$, there exists $T>0$ sufficiently small such that for all $0<t<T$ we have
$$\dAPk(x+t)\le \dAPk(x) + \varepsilon.$$ 

Let $\mathcal{U}=U_1\times...\times U_m$ be a product of finitely many compact abelian $p$-torsion groups. Let $f:\mathcal{U}\rightarrow[0,1]$ be measurable and suppose that $\int_{\mathcal{U}} f\,d\mu_{\mathcal{U}} \ge x$ and $S_{\mathcal{U}}^k(f) <\dAPk(x) + \varepsilon/2$. Consider the group $\mathcal{U}'$ being the same as $\mathcal{U}$ but with $U_1' = U_1\times \mathbb{F}_p^\mathbb{N}$.

For every $s\in[0,1]$, we can find a subset $B_s\subseteq \mathbb{F}_p^\N$ of density $s$ and set $$f_s(u,v):= \begin{cases} 1 & \text{if } v\in B_s\\ f(u) & \text{if } v\notin B_s\end{cases}$$ for $u\in \mathcal{U}$ and $v\in \mathbb{F}_p^\N$. Then $\int_{\mathcal{U}'} f_s~d\mu_{\mathcal{U}'} = (1-s)\int_{\mathcal{U}} f~d\mu_{\mathcal{U}} + s.$

By continuity of the expression above, we can find a continuous function $s:[0,1]\rightarrow [0,1]$ (mapping $0$ to itself) so that $\int_{\mathcal{U}'} f_{s(t)}~d\mu_{\mathcal{U}'} = \int_{\mathcal{U}} f~d\mu_{\mathcal{U}}+t$. In particular, we have $S^k_{\mathcal{U}'}(f_{s(t)})\geq\dAPk(x+t).$ On the other hand, $$\|f_{s(t)}-f\|_{L^1} \le s(t)\left(1-\int_{\mathcal{U}}f~d\mu_{\mathcal{U}}\right)\le   s(t)\cdot (1-x).$$ In particular, taking $T$ sufficiently small with respect to $\varepsilon$, $s$ and $x$, we can ensure that  $\|f_{s(t)}-f\|_{L^1}$ is arbitrarily small for all $0<t<T$.

Writing $f_{s(t)} = (f_{s(t)} - f) + f$ and opening the brackets gives that $S^k_{\mathcal{U}'}(f_{s(t)}) - S_{\mathcal{U}}^k(f)$ is a sum of $O(k)$-elements each bounded by $\|f_{s(t)}-f\|_{L^1}$. In other words
$$\left|S^k_{\mathcal{U}'}(f_{s(t)}) - S_{\mathcal{U}}^k(f) \right| = o_{k,T\rightarrow 0}(1).$$

Thus, taking $T$ sufficiently small (with respect to $f,s,\varepsilon,x$ and $k$) we see that for all $t<T$ we have
$$\dAPk(x+t)\leq \dAPk(x)+\varepsilon.$$

We will now use this to prove continuity at $x$. Let $\varepsilon>0$ and choose $T$ as in the claim above. By shrinking $T$ if necessary, we may further assume that $[x-T,x+T]\subseteq[0,1)$. Then for all $|t|<T$ we have: if $t>0$, then by monotonicity and the previous claim 
$$\dAPk(x)\le \dAPk(x+t)\le \dAPk(x) + \varepsilon,$$
and if $t<0$, by monotonicity and the previous claim with $x$ taking the place of $x+t$, we have
$$\dAPk(x) -\varepsilon \le \dAPk(x+t)\le \dAPk(x).$$ Either way, we see that 
$$|\dAPk(x)-\dAPk(y)|<\varepsilon$$ whenever $|x-y|<T$, which proves continuity in $x$.
\end{proof}

Finally, we show that the definition $\dAPk$ remains unchanged if the role of $f$ is replaced with an arbitrary subset $A$.
\begin{lemma}\label{function-set}
Let $k\ge 3$. For every $\delta\in[0,1]$ we set
    $$\dAPk'(\delta) = \inf_{U_1,\dots,U_m} \inf_{\mu(A)\ge \delta} S_{U_1,\dots,U_m}(1_A)$$ where the infimum is over all compact abelian $p$-torsion groups $U_1,\dots,U_m$ and subsets $A\subseteq \prod_{i=1}^m U_i$ satisfying $\mu(A)\ge \delta$, where $\mu$ is the Haar measure on $\mathcal{U}=\prod_{i=1}^m U_i$. Then $\dAPk'(\delta)=\dAPk(\delta)$.
\end{lemma}
\begin{proof}
    Since $1_A$ is a function taking values in $[0,1]$, and satisfying that $\int 1_A \,d\mu = \mu(A)\ge \delta$, we see that $\dAPk(\delta)\le \dAPk'(\delta)$. Therefore, it suffices to prove an inequality in the other direction. Equivalently, let $\delta>0$ and let $\varepsilon>0$ be arbitrarily small. We need to show that $\dAPk'(\delta)<\dAPk(\delta)+\varepsilon$.
    Since the continuous functions are dense in $L^1(\mu)$ and $\dAPk(\delta)$ is continuous in $\delta$, we can find compact abelian groups $U_1,\dots,U_m$, and a continuous function $f:\prod_{i=1}^m U_i\rightarrow [0,1]$ with $\int f \,d\mu> \delta$, such that \begin{equation}\label{locatef}S_{U_1,\dots,U_m}(f) <\dAPk(\delta)+\varepsilon/2.
    \end{equation}
    Extending $U_1,\dots,U_m$ if necessary and lifting $f$ accordingly, we may assume without loss of generality that the $U_i$ are infinite. Hence, we can embed $\mathbb{F}_p^\omega$ as a dense subgroup in each one of the $U_i$'s, and we see that $\mathrm{HP}_{0,1,\dots,k}(\mathbb{F}_p^\omega,\mathbb{F}_p^\omega,\dots,\mathbb{F}_p^\omega) \subseteq \mathrm{HP}_{0,1,\dots,k}(U_1,\dots,U_m)$ is also dense. Since $f$ is continuous, we have 
    $S_{U_1,\dots,U_m}(f) = \lim_{N\rightarrow\infty} S_{\mathbb{F}_p^N,\mathbb{F}_p^N,\dots,\mathbb{F}_p^N}(f)$. Applying Lemma \ref{randomV2}, we see that we can find a subset $E\subseteq \prod_{i=1}^m \mathbb{F}_p^\omega$, such that $d_{\Phi}(E)> \delta$ with respect to the classical F\o lner sequence $\Phi =(\prod_{i=1}^m \mathbb{F}_p^N)_{N\in\mathbb{N}}$, and
    $S_{U_1,\dots,U_m}(f) = \lim_{N\rightarrow\infty} S_{\mathbb{F}_p^N,\mathbb{F}_p^N,\dots,\mathbb{F}_p^N}(1_{E_N}),$ where $E_N = E\cap \prod_{i=1}^m \mathbb{F}_p^N$. Choose $N=N(\varepsilon)$ sufficiently large, such that $|E_N| > \delta p^{m\cdot N}$ and     \begin{equation}\label{locateN}
S_{\mathbb{F}_p^N,\mathbb{F}_p^N,\dots,\mathbb{F}_p^N}(1_{E_N}) < S_{U_1,\dots,U_m}(f) + \varepsilon/2.
    \end{equation} Write $\mathbb{F}_p^{\omega} = \mathbb{F}_p^N\times \mathbb{F}_p^{\omega\setminus N}$, and let $A$ denote the closure of $E_N\cdot \prod_{i=1}^m \mathbb{F}_p^{\omega\setminus N}$. Then $\mu(A) = |E_N|/p^{m\cdot N} > \delta$. On the other hand, it follows from the construction that
    $$S_{U_1,\dots,U_m}(1_A) \le S_{\mathbb{F}_p^N,\mathbb{F}_p^N,\dots,\mathbb{F}_p^N}(1_{E_N}).$$
    Combining this with \eqref{locatef} and \eqref{locateN} we see that
    $$\dAPk'(\delta) \le S_{U_1,\dots,U_m}(1_A)\le \dAPk(\delta)+\varepsilon,$$ as required.
\end{proof}


\bibliographystyle{abbrv}
\bibliography{bibliography}

\end{document}